\documentclass[11pt]{amsart}
\usepackage{amsmath, amssymb, amsthm, amsfonts,enumitem}
\usepackage[colorlinks = true, citecolor = blue]{hyperref}

\theoremstyle{plain}
\newtheorem{theorem}{Theorem}[section]
\newtheorem{corollary}[theorem]{Corollary}
\newtheorem{lemma}[theorem]{Lemma}
\newtheorem{proposition}[theorem]{Proposition}
\newtheorem{fact}[theorem]{Fact}
\newtheorem*{theorem*}{Theorem}
\newcounter{MainCorollaryCounter}
\newtheorem{MainCorollary}[MainCorollaryCounter]{Corollary}

\newcounter{MainTheoremCounter}
\newtheorem{MainTheorem}[MainTheoremCounter]{Theorem}

\newtheorem{MainSubTheorem}{Theorem}[MainTheoremCounter]

\theoremstyle{definition}
\newtheorem{definition}[theorem]{Definition}

\newtheorem*{setup}{Setup}

\newtheorem{example}[theorem]{Example}

\newtheorem*{conjecture*}{Conjecture}
\newcounter{MainConjectureCounter}
\newtheorem{MainConjecture}[MainConjectureCounter]{Conjecture}

\newtheorem{remark}[theorem]{Remark}

\newcommand{\modu}[1]{\overline{#1}}

\newcommand{\textdef}[1]{\textit{#1}}

\newcommand{\setdiff}{-}

\newcommand{\chevA}{\textbf{A}}
\newcommand{\chevB}{\textbf{B}}
\newcommand{\chevC}{\textbf{C}}
\newcommand{\chevD}{\textbf{D}}
\newcommand{\chevE}{\textbf{E}}
\newcommand{\chevF}{\textbf{F}}
\newcommand{\chevG}{\textbf{G}}

\newcommand{\complex}{\mathbb{C}}

\newcommand{\orbit}{\mathcal{O}}

\newcommand{\li}[1]{\ell_{#1}}

\DeclareMathOperator{\sym}{Sym} 
\DeclareMathOperator{\alt}{Alt} 
\DeclareMathOperator{\aff}{AGL} 
\DeclareMathOperator{\proj}{\mathbb{P}}

\DeclareMathOperator{\emorph}{End} 
 
\DeclareMathOperator{\Mouf}{\mathbb{M}} 
\DeclareMathOperator{\fp}{Fix} 
\DeclareMathOperator{\points}{\mathcal{P}} 
\DeclareMathOperator{\lines}{\mathcal{L}}

\DeclareMathOperator{\gtd}{gtd}
\DeclareMathOperator{\psl}{PSL} 
\DeclareMathOperator{\psp}{PSp} 
\DeclareMathOperator{\pssl}{(P)SL} 
\DeclareMathOperator{\ssl}{SL}
\DeclareMathOperator{\pgl}{PGL} 
\DeclareMathOperator{\agl}{AGL} 
\DeclareMathOperator{\gl}{GL} 
 
\DeclareMathOperator{\rk}{rk}

\begin{document}
\title[Recognizing $\pgl_3$ via generic $4$-transitivity]{Recognizing $\pgl_3$ via generic $4$-transitivity}
\author{Tuna Alt\i{}nel}
\address{Universit\'e de Lyon \\
Universit\'e Claude Bernard Lyon 1\\
CNRS UMR 5208\\
Institut Camille Jordan\\
43 blvd du 11 novembre 1918\\
F-69622 Villeurbanne cedex\\
France}
\email{altinel@math.univ-lyon1.fr}
\author{Joshua Wiscons}
\address{Department of Mathematics\\
Hamilton College\\
Clinton, NY 13323, USA}
\email{jwiscons@hamilton.edu}
\thanks{This material is based upon work supported by the National Science Foundation under grant No. OISE-1064446.}
\begin{abstract}
We show that the only transitive and generically $4$-transitive action of a group of finite Morley rank on a set of Morley rank $2$ is the natural action of $\pgl_3$ on the projective plane. 
\end{abstract}
\maketitle

\section{Introduction}\label{sec.Intro}
The notion of generic $n$-transitivity was first introduced in the context algebraic groups by Popov in \cite{PoV07} and later by Borovik and Cherlin in \cite{BoCh08} for groups of finite Morley rank. Morley rank is a model theoretic notion of dimension inspired by algebraic geometry, and for an algebraically closed field (considered in the language of fields), the Morley rank of a Zariski closed set is equal to its Zariski dimension. As such, the algebraic groups over algebraically closed fields are the primary examples of groups of finite Morley rank. In fact, it was conjectured  by Cherlin and Zil'ber over three decades ago that every \emph{simple} group of finite Morley rank is  an algebraic group over an algebraically closed field; the conjecture is still open. 

We will work in the finite Morley rank category and defer to \cite{PoB87}, \cite{BoNe94}, and \cite{ABC08} for the necessary background. The question we address is interesting even when restricted to the algebraic context, and the reader without knowledge of groups of finite Morley rank is encouraged to, if necessary, translate ``rank'' to ``dimension'' and ``definable'' to ``constructible.''  The main point is that, although we have a notion of dimension, we have no topology. In what follows, rank  always refers to Morley rank.

\begin{definition}\label{def.GenericTrans}
A definable action of a group of finite Morley rank $G$ on a definable set $X$ is said to be \textdef{generically $n$-transitive} if $G$ has an orbit $\orbit$ on $X^n$ such that the rank of $X^n - \orbit$ is strictly less than the rank of $X^n$. In this case, the action is generically \textdef{sharply} $n$-transitive if $G$ acts regularly on $\orbit$, i.e. if the stabilizer of any $n$-tuple from $\orbit$ is trivial.
\end{definition}

Note that by \cite{TiJ52} and \cite{HaM54} \emph{ordinary} sharp $4$-transitivity does not exist on any infinite set, and if one works in the algebraic category, \emph{ordinary} $4$-transitivity does not exist either, see \cite{KnF83}. In contrast, the action of $\gl_n(\complex)$ on $\complex^n$ is generically sharply $n$-transitive for each $n$, and one has the feeling that this is the right notion of transitivity in this setting. We address the following conjecture about natural limits to generic $n$-transitivity.

\begin{MainConjecture}[\protect{\cite[Problem~9]{BoCh08}}]\label{conj.ProbPGL}
If $(X,G)$ is a transitive and generically $(n+2)$-transitive permutation group of finite Morley rank with $G$ connected and $\rk X = n$, then
$(X,G)$ is equivalent to  $(\proj^{n}(K),\pgl_{n+1}(K))$ for some algebraically closed field $K$.
\end{MainConjecture}

The conjecture holds when $X$ has rank $1$ as a consequence of the much stronger Fact~\ref{fact.Hru}, but relatively little is known when $X$ has rank $2$. If the action is by \emph{automorphisms} on an abelian group of rank $2$, then \cite{DeA09} shows that one does not exceed generic $2$-transitivity. The main result in the general rank $2$ setting is due to Gropp in \cite{GrU92} where the author shows that if one assumes generic \emph{sharp} $n$-transitivity then $n$ is at most $5$. Note that in the algebraic setting one should get a lot of mileage out of \cite{PoV07} (regardless of the rank of $X$ -- though only in characteristic $0$) when combined with the O'Nan-Scott type theorem of Macpherson and Pillay in \cite{MaPi95}. Here, we prove the conjecture in full generality for sets of rank $2$.

\begin{MainTheorem}\label{thm.A}
If $(X,G)$ is a transitive and generically $4$-transitive permutation group of finite Morley rank with $\rk X = 2$, then $(X,G)$  is equivalent to $(\proj^2(K),\pgl_{3}(K))$ for some algebraically closed field $K$.
\end{MainTheorem}

Combining Theorem~\ref{thm.A} with Fact~\ref{fact.Hru} and Lemma~\ref{lem.PrimitiveBoundG}, one immediately obtains the following corollary.
\begin{MainCorollary}
If $G$ is a simple group of finite Morley rank  with a definable subgroup of corank $2$, then $\rk G \le 8$, and if $\rk G = 8$, then $G\cong \pgl_{3}(K)$ for some algebraically closed field $K$.
\end{MainCorollary}

We mention that although many parts of our analysis appear to be tightly tied to the case of rank $2$, several pieces seem rather general (or at least generalizable), e.g.  our preliminary recognition theorem, Proposition~\ref{prop.Recognition}. Further, we illustrate the importance of the so-called $\Sigma$-groups, 
see Subection~\ref{subsec.Sigma}. These groups allow one to identify the prospective Weyl group, and as such, they should be an integral piece to future work on Conjecture~\ref{conj.ProbPGL}.

Focusing on the case where $G$ is connected, we prove Theorem~\ref{thm.A} in two steps. After some preparation, we establish the theorem under the additional restriction of generic sharp$^\circ$ $n$-transitivity; this is Theorem~\ref{thm.Aone}. By generic \textdef{sharp$^\circ$} $n$-transitivity, we mean a generically $n$-transitive action for which $\rk G = \rk X^n$, i.e.  the \emph{connected component} of the stabilizer of a generic $n$-tuple in $X^n$ is trivial. It then more-or-less remains to show that generically $4$-transitive actions for which a generic $4$-point stabilizer has rank less than $2$ are in fact generic sharply$^\circ$ $4$-transitive; this is Theorem~\ref{thm.Atwo}. Everything is properly glued together in Section~\ref{sec.A} where we also give the reduction of Theorem~\ref{thm.A} to the case where $G$ is connected.

The paper is organized as follows. Section~\ref{sec.Examples} gives some basic examples of generic multiple transitivity, and Section~\ref{sec.Background} lays out a handful of    definitions and results on groups of small Morley rank. Both sections are brief. The general permutation group-theoretic terminology and tools for our analysis are given in Section~\ref{sec.PermGroups}. Section~\ref{sec.TargetRecognition} is where our proof of Theorem~\ref{thm.A} really begins. This section is devoted to proving a natural approximation to Theorem~\ref{thm.A}, namely Proposition~\ref{prop.Recognition}; the approach is very geometric. In Section~\ref{sec.Aone}, we prove Theorem~\ref{thm.Aone}. It seems worth noting that our analysis in Section~\ref{sec.Aone} first establishes the bound of $4$ on the degree of generic transitivity before showing that Proposition~\ref{prop.Recognition} applies. The bound comes entirely from the presence of a large $\Sigma$-group. Section~\ref{sec.Atwo} contains the proof of Theorem~\ref{thm.Atwo}, and Section~\ref{sec.A} tidies everything up. 

\section{Examples}\label{sec.Examples}
We briefly give some examples of generic multiple transitivity. Throughout this section $K$ denotes an algebraically closed field definable in some ambient structure of finite Morley rank. This includes the algebraic setting where the ambient structure is just $K$ considered in the language of fields. Fix a positive integer $n$, and let $V$ be the vector space $K^n$.  

\begin{example}
The natural action of $\gl_n(K)$ on $V$ is generically sharply $n$-transitive; the generic orbit in $V^n$ consists of the ordered bases for $V$. 
\end{example}

Of course, similar statements hold for $\agl_n$ and $\pgl_n$.

\begin{example}
The natural actions of $\agl_n(K)$ on $V$ and $\pgl_n(K)$ on $\proj(V)$ are both generically sharply $(n+1)$-transitive.
\end{example}

The situation changes dramatically if one moves to, say, $\psp_n$ where the degree of generic transitivity becomes bounded independent of $n$. This is just one instance of Fact~\ref{fact_Popov}, which provides a bit of evidence for Conjecture~\ref{conj.ProbPGL}.

\begin{example}
If $n=2m$, then the natural action of $\psp_n(K)$ on $\proj(V)$ is generically $3$-transitive but not generically $4$-transitive. To see this, fix a symplectic basis $e_1,f_1,\ldots,e_m,f_m$ for $V$, i.e. $[e_i,e_j] = [f_i,f_j] = 0$ and $[e_i,f_j] = \delta_{ij}$  for all $1\le i,j \le m$. Now, one can compute that the orbit of $(\langle e_1\rangle, \langle f_1\rangle, \langle e_1+f_1+e_2\rangle)$ is generic in $(\proj(V))^3$ and that the stabilizer of the triple has no chance to have a generic orbit  on $\proj(V)$. 

To better see why generic $4$-transitivity fails, let $H$ be the stabilizer of $\langle e_1\rangle$ and $\langle f_1\rangle$ in $\psp_n(K)$, and let $\orbit$ be the generic orbit of $H$ on $\proj(V)$. If the action of $\psp_n(K)$ on $\proj(V)$ is generically $4$-transitive, then the action of $H$ on $\orbit$ must be generically $2$-transitive. Setting $W:=\langle e_1\rangle\oplus \langle f_1\rangle$, we have an $H$-invariant decomposition $V=W\oplus W^\perp$, so $\langle e_1\rangle$ and $\langle f_1\rangle$ determine an $H$-invariant  projection $\pi:\proj(V)\rightarrow \proj(W)$. Thus, the action of $H$ on $\pi(\orbit)$ is a quotient of the action of $H$ on $\orbit$, so the degree of generic transitivity of the former action is at least as large as that of the latter, see Lemma~\ref{lem.QuotientGenericTrans}. As $\proj(W)$ is a $1$-dimensional projective space and $H$ fixes two points, $H$ acts on $\proj(W)$ as a torus, so the action of $H$ on $\pi(\orbit)$ cannot be generically $2$-transitive. 
\end{example}

\begin{fact}[{\cite[Theorem~1]{PoV07}}]\label{fact_Popov}
If $G$ is a simple algebraic group over an algebraically closed field of characteristic $0$, then the maximum degree of generic transitivity among all nontrivial actions of $G$ on  irreducible algebraic varieties, denoted $\gtd(G)$, is given by the following table.
\begin{center}\small 
\begin{tabular}{c||c|c|c|c|c|c|c|c|c}
\textnormal{type of} $G$ & $\chevA_n$ & $\chevB_n, n \ge 3$ & $\chevC_n, n \ge 2$ & $\chevD_n, n \ge 4$ & $\chevE_6$ & $\chevE_7$ & $\chevE_8$ & $\chevF_4$ & $\chevG_ 2$\\ \hline
$\gtd{G}$ & $n +2$ & $3$ & $3$ & $3$ & $4$ & $3$ & $2$ & $2$ & $2$
\end{tabular}
\end{center}
\end{fact}

For more on the algebraic setting, we refer the reader  to \cite{PoV07}. Finally, we consider products of known actions; here we find it rather satisfying that we are not restricted to the algebraic category.

\begin{example}\label{exam.Products}
If $G_1$ acts generically $t_1$-transitively on $X_1$ and $G_2$ acts generically $t_2$-transitively on $X_2$, then the coordinatewise action of $G_1\times G_2$ on $X_1\times X_2$ is generically $t$-transitive with $t = \min(t_1,t_2)$. In particular, if $L$ is an algebraically closed field, possibly of characteristic different from that of $K$, with $K$ and $L$ definable in some ambient structure of finite Morley rank, then the action of $\gl_n(K) \times \gl_m(L)$ on $K^n \times L^m$ is generically $\min(n,m)$-transitive.
\end{example}

\section{Groups of small rank}\label{sec.Background}
As mentioned in the introduction, background on groups of finite Morley rank can be found in  \cite{PoB87}, \cite{BoNe94}, and \cite{ABC08}. In this section, we collect some specialized results about groups of small Morley rank. We first need a few definitions. It should be mentioned that our definition of a unipotent group is definitely not standard.

\begin{definition}
Let $G$ be a group of finite Morley rank. Then
\begin{enumerate}
\item $G$ is called a \textdef{decent torus} if $G$ is divisible, abelian, and equal to the definable hull of its torsion subgroup, 
\item $G$ is called a \textdef{good torus} if every definable subgroup of $G$ (including $G$) is a decent torus, and
\item $G$ is said to be \textdef{unipotent} if $G$ is connected, nilpotent, and does not contain a nontrivial decent torus. 
\end{enumerate}
\end{definition}

\subsection{Groups of rank at most $3$}
Most of the results in this subsection can be found in Cherlin's paper \cite{ChG79} though the first is due to Reineke.

\begin{fact}[\cite{ReJ75}]\label{fact.rankOneGroups}
If $A$ is a connected group of rank $1$, then $A$ is either a divisible abelian group or an elementary abelian $p$-group for some prime $p$.
\end{fact}

\begin{fact}[\cite{ChG79}]\label{fact.rankTwoGroups}
Let $B$ be a connected group of rank $2$. Then $B$ is solvable. If $B$ is nilpotent and nonabelian, then $B$ has exponent $p$ or $p^2$ for some prime $p$. If $B$ is nonnilpotent, then
\begin{enumerate}
\item $B = B' \rtimes T$ with $T$ a good torus containing $Z(B)$,  
\item $B/Z(B) \cong K^+ \rtimes K^\times$ for some algebraically closed field $K$, and 
\item every definable automorphism of $B$ of finite order is inner.
\end{enumerate}
\end{fact}

The following fact follows almost immediately from the previous one, but a proof can be found in \cite{WiJ14a}.

\begin{fact}\label{fact.RankTwoAbelian}
If $B$ is a connected group of rank $2$, then any one of the following implies that $B$ is abelian:
\begin{enumerate}
\item $B$ normalizes a nontrivial decent torus,
\item $B$ contains two distinct unipotent subgroups of rank $1$, or
\item $B$ is nilpotent and contains two distinct definable connected subgroups of rank $1$.
\end{enumerate} 
\end{fact}

Recall that the \textdef{Fitting subgroup} of a group $G$ is the subgroup $F(G)$ generated by all normal nilpotent subgroups, and in a group of finite Morley rank, the Fitting subgroup is nilpotent and definable. A group is called \textdef{quasisimple} if it is perfect, and modulo its center, it is simple. Also, we say that a group of finite Morley rank is a \textdef{bad group} if it is a connected nonsolvable group such that every proper definable subgroup is nilpotent. It is unknown if bad groups exist.

\begin{fact}[\cite{ChG79}]\label{fact.rankThreeGroups}
Let $G$ be a connected group of rank $3$.
\begin{enumerate}
\item If $\rk F(G) \ge 1$, then $G$ is solvable.
\item If $\rk F(G) = 0$, then either
\begin{enumerate}
\item $G$ is a quasisimple bad group, or
\item $G\cong \pssl_2(K)$ for some algebraically closed field $K$. 
\end{enumerate}
\end{enumerate}
\end{fact}

\subsection{Groups of rank $4$}
Fact~\ref{fact.genericSharpTwoTransRankTwo}, which is a corollary of the following fact, will be used in Subsection~\ref{subsec.TransBound} where it is more-or-less responsible for our proof of Theorem~\ref{thm.Aone} getting off the ground. Note that for a group $G$ with subgroups $A$ and $B$, we write $G=A*B$ if $A$ and $B$ commute and generate $G$, i.e. $G$ is the \textdef{central product} of $A$ and $B$.

\begin{fact}[{\cite[Corollary~A]{WiJ14a}}]\label{fact.rankFourGroups}
Let $G$ be a connected group of rank $4$.
\begin{enumerate}
\item If $\rk F(G) \ge 2$, then $G$ is solvable.
\item If $\rk F(G) = 1$, then either
\begin{enumerate}
\item $G$ is a quasisimple bad group, or
\item $G = F(G) * Q$ for some quasisimple subgroup $Q$ of rank $3$.
\end{enumerate}
\item If $\rk F(G) = 0$, then either
\begin{enumerate}
\item $G$ is a quasisimple bad group, or
\item $G$ has a normal quasisimple bad subgroup of rank $3$. 
\end{enumerate}
\end{enumerate}
In particular, $\rk F(G) \ge 1$ whenever $G$ has an involution.
\end{fact}

\begin{fact}[{\cite[Corollary~B]{WiJ14a}}]\label{fact.genericSharpTwoTransRankTwo}
If $G$ is a connected nonsolvable group of rank $4$ acting faithfully, definably, transitively, and generically $2$-transitively on a definable set of rank $2$, then there is an algebraically closed field $K$ for which $G =Z(G) \cdot Q$ with  $Z(G)\cong K^\times$ and $Q\cong \pssl_2(K)$. 
\end{fact}

\section{Permutation groups}\label{sec.PermGroups}
We now give the essential definitions and some permutation group theoretic background for our study of generically $n$-transitive actions. Our main references for this material are \cite{BoCh08} and \cite{WiJ14a}.

\begin{definition}
A pair $(X,G)$ is called a \textdef{permutation group} if $G$ is a group with a (fixed) faithful action on the set $X$. We say that $(X,G)$ has \textdef{finite Morley rank} if $G$, $X$, and the action of $G$ on $X$ are all definable in some ambient structure of finite Morley rank. Additionally, if $(X,G)$ is a permutation group of finite Morley rank that is generically $n$-transitive but not generically $(n+1)$-transitive, we say that $n$ is the \textdef{degree of generic transitivity} and write $\gtd(X,G)=n$.
\end{definition}

Before going further, it will be helpful to make some basic remarks about generic $n$-transitivity. These remarks will frequently be used without explicit 
mention or reference.

\begin{remark}\label{generalremarksongenerictransitivity} Let $(X,G)$ be a generically $n$-transitive permutation group of finite Morley rank.
\begin{enumerate}
\item As a consequence of the additivity of Morley rank in the context of groups, $(X,G)$ is easily seen to be generically $(n-1)$-transitive. 
\item If $x$ is in the generic orbit of $G$ on $X$, then $G_x$ is generically $(n-1)$-transitive on $X$, again using the additivity of the rank.
\item\label{generalremarksongenerictransitivity.ps} If $X$ is infinite and $n\geq2$, then the previous point implies that the $1$-point stabilizers, for points in the generic orbit, are infinite.
\item By restricting the action to the generic orbit of $G$ on $X$, one obtains a transitive action that is still generically $n$-transitive. This yields a natural reduction of many arguments to the transitive case, but when we want to understand the structure imposed on the original $X$,  we will explicitly assume transitivity at the outset.
\end{enumerate}
\end{remark}

\subsection{Connectedness results}
We tend to focus on transitive permutation groups $(X,G)$ for which $G$ is connected. In this case, $X$ is also connected, i.e. of degree $1$, by general principles. If $G$ is not connected, we may still be able to use the following fact to establish connectedness of $X$. 

\begin{fact}[{\cite[Lemma~1.8(3)]{BoCh08}}]\label{fact.GenericTwoImpliesConnected}
If  an infinite  permutation group of finite Morley rank $(X,G)$ is generically $2$-transitive, then $X$ is connected. 
\end{fact}

Note that when $X$ is connected the definition of generic $n$-transitivity reduces to $G$ having an orbit $\orbit$ on $X^n$ with $\rk \orbit = \rk X^n$. 
One frequently used consequence of $X$ being connected is that, in this case, fixing a generic subset of $X$ is equivalent to fixing all of $X$.

\begin{fact}[{\cite[Lemma~1.6]{BoCh08}}]\label{fact.FixGenericSet}
If $(X,G)$ is an infinite transitive permutation group of finite Morley rank with $X$ connected, then only the identity fixes a generic subset of $X$. 
\end{fact}

This fact appears in a variety of disguises. We now highlight an important one, which will be used frequently in the sequel.

\begin{lemma}\label{lem.HKlemma}
Let $(X,G)$ be an infinite transitive permutation group of finite Morley rank with $X$ connected. 
Assume that $H$ and $K$ are definable subgroups of $G$ such that $H$ that fixes some $x\in X$ while the orbit of $K$ on $x$ is generic. 
\begin{enumerate} 
\item If $K \le N(H)$, then $H=1$.
\item If $H\le N(K)$, then $H$ acts faithfully on $K$.
\end{enumerate}
\end{lemma}
\begin{proof}
First suppose that $K \le N(H)$. Then $H$ fixes all of the $K$-conjugates of $x$. Since the orbit of $K$ on $x$ is generic, $H$ fixes a generic subset of $X$, so by Fact~\ref{fact.FixGenericSet}, $H = 1$. Next assume that $H\le N(K)$. Certainly $K \le N(C_H(K))$, so applying the first point with $C_H(K)$ in place of $H$, we find that $C_H(K)=1$.
\end{proof}

We now change the focus to $G$ and mention some connectedness results for point stabilizers.

\begin{fact}[{\cite[Lemma~3.5]{WiJ14a}}]\label{fact.ConnectedPS}
Assume that $(X,G)$ is a transitive permutation group of finite Morley rank with $G$ connected. If some definable subgroup of $G$ has a regular and generic orbit on $X$, then all $1$-point stabilizers are connected.  
\end{fact}

Our main application of the previous fact requires a definition.

\begin{definition}
Let $(X,G)$ be a permutation group of finite Morley rank. We say that $(x_1,\ldots,x_n) \in X^n$ 
is in \textdef{general position} if the orbit of $G$ on $X^n$ containing $(x_1,\ldots,x_n)$ is of maximal rank among all such orbits. The stabilizers of $n$-tuples in general position are called \textdef{generic $n$-point stabilizers}.
\end{definition}

For generically $n$-transitive actions, the following lemma shows that in the definition of a tuple in general position we can replace the tuple by the set of its coordinates. We will do this frequently in the sequel. Lemma~\ref{lem.PermuteGeneric} also shows that generically $n$-transitive groups contain a section isomorphic to the full symmetric group $\sym(n)$; this will be exploited in Subsection~\ref{subsec.Sigma}. 

\begin{lemma}\label{lem.PermuteGeneric}
Let $(X,G)$ be a generically $n$-transitive permutation group of finite Morley rank, and $(x_1,\ldots,x_n) \in X^n$ in general
position. Then every permutation of $(x_1,\ldots,x_n)$ is in general position.
\end{lemma}
\begin{proof}
Permuting coordinates induces a rank-preserving definable bijection of $X^n$. As a result, the image under such a bijection of the orbit containing $(x_1,\ldots,x_n)$  also has a non-generic complement. Thus, the image and the initial orbit have nonempty intersection. Since the image is also an orbit, they are in fact equal. 
\end{proof}

\begin{fact}[{\cite[Lemma~3.7]{WiJ14a}}]\label{fact.GenNTransAbelian}
Let $(X,G)$ be an infinite transitive permutation group of finite Morley rank with $G$ connected. If $n:=\gtd(X,G)\ge 2$ and a generic $(n-1)$-point stabilizer is abelian-by-finite, then 
\begin{enumerate}
\item  the generic $k$-point stabilizers are connected for all $k\le n-1$,
\item the generic $(n-1)$-point stabilizers are self-centralizing in $G$, 
\item the action is generically \emph{sharply} $n$-transitive, and
\item $G$ is centerless.
\end{enumerate}
\end{fact}

Actually, the previous fact has been stated in a slightly expanded form, but it is easy to deduce from the original. Indeed, all but the third point come directly from \cite[Lemma~3.7]{WiJ14a} and imply that a generic $(n-1)$-point stabilizer is abelian. Thus, if $H$ is a generic $(n-1)$-point stabilizer, then $H$ must have a generic orbit on $X$, and as $H$ is abelian, Fact~\ref{fact.FixGenericSet} implies that $H$ must act regularly on the generic orbit. This implies the third point. We now mention a pair of reductions to connected groups. 

\begin{lemma}\label{fact.ConnectedCompNTrans}
If $(X,G)$ is a permutation group of finite Morley rank with $X$ connected. Then $(X,G)$ is $n$-transitive if and only if $(X,G^\circ)$ is $n$-transitive.
\end{lemma}
\begin{proof}
Only the forward direction is nontrivial. Assume that $(X,G)$ is $n$-transitive. Let $X^{(n)}$ be the subset of $X^n$ consisting of those tuples for which all entries are pairwise distinct. We need to show that $G^\circ$ acts transitively on $X^{(n)}$. Note that $X^{(n)}$ is a definable subset of $X^n$, and that $X^n -  X^{(n)}$ is a union of a finite number of definable subsets each of rank strictly less than the rank of $X^n$. Thus $X^{(n)}$ and $X^n$ both have the same rank and degree. Most importantly, we find that $X^{(n)}$ has degree $1$, since $X$ was assumed to be connected. Now, $G$ is transitive on $X^{(n)}$, so every orbit of $G^\circ$ on $X^{(n)}$ has full rank in $X^{(n)}$. As $X^{(n)}$ has degree $1$, $G^\circ$ has a single orbit on $X^{(n)}$.
\end{proof}

\begin{fact}[\protect{\cite[Lemma~1.9]{BoCh08}}]\label{fact.ConnectedCompGenericTrans}
If $(X,G)$ is a transitive permutation group of finite Morley rank with $X$ connected, then $\gtd(X,G) = \gtd(X,G^\circ)$.
\end{fact}

We close this subsection with a useful lemma that relates multiple generic transitivity to the generation of connected groups.

\begin{lemma}\label{lem.GenericTwoTransPS}
If $(X,G)$ is a generically $2$-transitive permutation group of finite Morley rank with $x,y\in X$  in general position, then $G^\circ = \langle (G_x)^\circ, (G_y)^\circ\rangle$.
\end{lemma}
\begin{proof}
Identifying $X$ with the (right) coset space $G_x\backslash G$, our assumptions imply that the orbit of $G_y$ on the coset $G_x$ has full rank in $G_x\backslash G$, i.e. that $\rk G_x \backslash (G_xG_y) = \rk G_x \backslash G$. Thus, $\rk G = \rk G_xG_y$, so $(G_x)^\circ$ and $(G_y)^\circ$ generate $G^\circ$.
\end{proof}

\subsection{Primitivity, quotients, and covers}
When there is a high degree of generic transitivity relative to the rank of the set being acted upon, one expects to encounter some sort of primitivity. This is, of course, implied by Conjecture~\ref{conj.ProbPGL}. 

\begin{definition}
Assume that a group $G$, a set $X$, and an action of $G$ on $X$ are all definable in some ambient structure. The action is said to be \textdef{definably primitive} if every definable (with respect to the ambient structure) $G$-invariant equivalence relation is either trivial or universal; where as, the action called \textdef{virtually definably primitive} if every definable $G$-invariant equivalence relation either has finite classes or finitely many classes. 
\end{definition}

For permutation groups of finite Morley rank, it turns out that definable primitivity often coincides with ordinary primitivity.

\begin{fact}[\protect{\cite[Proposition~2.7]{MaPi95}}]\label{fact.DefPrimImpliesPrim}
If $(X,G)$ is a definably primitive permutation group of finite Morley rank with infinite point stabilizers, then  $(X,G)$ is primitive.
\end{fact}

For transitive actions, there is an inclusion-preserving bijection from the set of $G$-invariant equivalence relations on $X$ to the set of subgroups of $G$ containing a fixed point stabilizer $G_x$. Specifically, if $\sim$ is a $G$-invariant equivalence relation on $X$ and $\modu{x}$ is the class of $x$, then the bijection takes $\sim$ to the \emph{setwise} stabilizer $G_{\modu{x}}$. The inverse takes an overgroup $H$ of $G_x$ to the equivalence relation whose classes are the $G$-translates of the orbit $xH$.
Thus, a transitive action is primitive if and only if a point stabilizer is a maximal subgroup of $G$, and one also obtains the following analogous characterizations of definable and virtually definable primitivity. 

\begin{fact}[\protect{\cite[Lemma~1.13]{BoCh08}}]\label{fact.characterizeDefnPrimi}
Let $(X,G)$ be a transitive permutation group definable in some ambient structure, and fix $x\in X$. Then 
\begin{enumerate}
\item $(X,G)$ is definably primitive if and only if $G_x$ is a maximal definable subgroup of $G$, and 
\item\label{fact.characterizeDefnPrimi.vdp} $(X,G)$ is virtually definably primitive if and only if for every definable subgroup $H$ containing $G_x$ either $|G:H|$ or $|H:G_x|$ is finite. 
\end{enumerate}
\end{fact}

If $(X,G)$ is a permutation group and $\sim$ is a $G$-invariant equivalence relation on $X$, then $G$ acts naturally on $X/{\sim}$, and we call $(X/{\sim},G)$ a \textdef{quotient} of $(X,G)$. Note that a quotient may no longer be faithful. As mentioned above, for $(X,G)$ transitive, the definable quotients of $(X,G)$ correspond to the definable subgroups of $G$ containing $G_x$. By moving to any proper definable overgroup of $G_x$ of maximal rank, one sees that every transitive permutation group of finite Morley rank has a virtually definably primitive quotient, but in fact we can often find a quotient that is definably primitive. 

\begin{fact}[\protect{\cite[Lemma~1.18]{BoCh08}}]\label{fact.primQuotient}
Let $(X,G)$ be a transitive permutation group of finite Morley rank. Then 
\begin{enumerate}
\item $(X,G)$ has a nontrivial virtually definably primitive quotient, and
\item if $(X,G)$ has a nontrivial virtually definably primitive quotient with infinite point stabilizers, then $(X,G)$ has a nontrivial definably primitive quotient.
\end{enumerate}
\end{fact}

We will often pass to definably, or at least virtually definably, primitive quotients, and when we are in a context of generic $n$-transitivity, the following point is essential. 

\begin{lemma}[\protect{\cite[Lemma~6.1]{BoCh08}}]\label{lem.QuotientGenericTrans}
Let $(X,G)$ be a transitive permutation group of finite Morley rank with $n:=\gtd(X,G)$. If $x_1,\ldots,x_n\in X$ are in general position and $(\modu{X},G)$ is any infinite definable quotient, then $\gtd(\modu{X},G)\ge n$, and the images of $x_1,\ldots,x_n$ in $\modu{X}$ are in general position.
\end{lemma}
\begin{proof}
The transitivity of $G$ on $X$ easily transfers to $(\modu{X},G)$, so the classes in $\modu{X}$ have constant rank. If $\orbit\subset X^n$ is the orbit of $G$ on  $(x_1,\ldots,x_n)$, then it is not hard to see that the image of $\orbit$ in $\modu{X}^n$ is a $G$-orbit in $\modu{X}^n$ of full rank. As we need only address when $n\ge2$, we can take $X$, hence $\modu{X}$, to be connected, so $(\modu{X},G)$ is generically $n$-transitive.
\end{proof}

On occasion we will want to assume that point stabilizers are connected. This amounts to passing to a finite cover.

\begin{definition}
If $G$ is a definable group acting on definable sets $X$ and $\widehat{X}$, then we call $\widehat{X}$ a \textdef{definable cover} of $X$ if there is a 
definable surjective $G$-invariant map  $\pi:\widehat{X}\rightarrow X$, i.e. $(X,G)$ is equivalent to a quotient of $(\widehat{X},G)$. If the fibers of $\pi$ are finite, we say that $\widehat{X}$ is a \textdef{finite cover} of $X$.
\end{definition}

Moving from a transitive permutation group to a transitive finite cover amounts to moving to a subgroup of finite index in a point stabilizer. In the case that $(X,G)$ is primitive, any nontrivial cover will certainly not be primitive and may even fail to be virtually definably primitive. However, we do have the following. Note that we assume $G$ to be connected. 

\begin{lemma}\label{lem.CoverGenericTrans}
Let $(X,G)$ be an infinite transitive permutation group of finite Morley rank with $G$ connected and $n:=\gtd(X,G)$. If $x_1,\ldots,x_n\in X$ are  in general position and $(\widehat{X},G)$ is any definable transitive finite cover of $(X,G)$, then $\gtd(\widehat{X},G) = n$, and any $\hat{x}_1,\ldots,\hat{x}_n\in\widehat{X}$ that project to $x_1,\ldots,x_n$ are in general position.
\end{lemma}
\begin{proof}
As the covering map has finite fibers, $X$ and $\widehat{X}$ have the same rank. Further, both sets are connected since $G$ is connected and the actions are transitive. Now, if $\orbit\subset \widehat{X}^n$ is the orbit of $G$ on $(\hat{x}_1,\ldots,\hat{x}_n)$, then $\orbit$ projects to the generic orbit of $G$ on $X^n$. Since the fibers of the projection from $\widehat{X}$ to $X$, hence from $\widehat{X}^n$ to $X^n$, are finite, we find that $\orbit$ has the desired rank. 
\end{proof}

Early in our proof of Theorem~\ref{thm.Aone}, we will show that the generic $(n-1)$-point stabilizers are abelian-by-finite. We will also be in a virtually definably primitive context, so we will have access to the following result.  

\begin{corollary}\label{cor.GenNTransAbelian}
Let $(X,G)$ be a transitive virtually definably primitive permutation group of finite Morley rank with $G$ connected. If $n:=\gtd(X,G)\ge 2$ and a generic $(n-1)$-point stabilizer is abelian-by-finite, then the action is primitive.
\end{corollary}
\begin{proof}
If we pass to a definably primitive quotient of $(X,G)$ with kernel $K$, then as the classes in the quotient are finite, $K^\circ = 1$ (using the fact that connected groups act trivially on finite sets). Thus, $K$ is finite and hence central since $G$ is connected. By Fact~\ref{fact.GenNTransAbelian}, $K=1$. Further, Fact~\ref{fact.GenNTransAbelian} applies to the quotient, so the $1$-point stabilizers in the quotient are connected. As there is no kernel, the $1$-point stabilizers from the quotient coincide with the original $1$-point stabilizers, so $(X,G)$ is definably primitive. By Remark \ref{generalremarksongenerictransitivity}\eqref{generalremarksongenerictransitivity.ps}, generic $2$-transitivity 
ensures that the point stabilizers are infinite, so the action is primitive by Fact~\ref{fact.DefPrimImpliesPrim}.
\end{proof}

In the situation where the $1$-point stabilizers are abelian, we obtain the following characterization of the $1$-dimensional affine group. In our proof of Theorem~\ref{thm.Aone}, this can be used to show that a generic $(n-2)$-point stabilizer is \emph{not} virtually definably primitive.

\begin{fact}[{\cite[Proposition~3.8]{WiJ14a}}]\label{fact.VprimRankAbelian}
Let $(X,G)$ be an infinite transitive and virtually definably primitive permutation group of finite Morley rank with abelian point stabilizers. If $G$ is connected and the point stabilizers have rank at least $\rk X$, then $(X,G) \cong (K,\aff_1(K))$ for some algebraically closed field $K$. 
\end{fact}

Finally, we mention that, in the presence of primitivity, there is an essential connection between the degree of generic transitivity and the rank of the group. The following fact is stated in a slightly stronger form than the lemma it references; the stated form is clear from the original proof.

\begin{fact}[\protect{\cite[Proposition~2.3]{BoCh08}}]\label{fact.PrimitiveBoundG}
If $(X,G)$ is a primitive permutation group of finite Morley rank with $\rk X = r$, then \[\rk G \le r\cdot \gtd(X,G) + r(r-1)/2.\]
\end{fact}

We will need the following very slight modification of the previous fact.

\begin{lemma}\label{lem.PrimitiveBoundG}
If $(X,G)$ is a transitive virtually definably primitive permutation group of finite Morley rank with $G$ connected and $\rk X = r$, then \[\rk G \le r\cdot \gtd(X,G) + r(r-1)/2.\]
\end{lemma}
\begin{proof}
If $r = 0$, there is nothing to show, so assume that $X$ is infinite. If $(X,G)$ has finite point stabilizers, then $\rk G = r$, and we are done since $\gtd(X,G)\ge 1$. Thus, we assume that $(X,G)$ has infinite point stabilizers, so Fact~\ref{fact.primQuotient} yields  an infinite definably primitive quotient $(\modu{X},G)$ with kernel $K$. As the classes in the quotient are finite, the kernel is also finite, so the ranks of $X$ and $G$ coincide with those of $\modu{X}$ and $G/K$. Further, Lemma~\ref{lem.CoverGenericTrans} shows that $\gtd(\modu{X},G/K) = \gtd(X,G)$, so we may assume that $(X,G)$ is definably primitive. Again, if $(X,G)$ has finite point stabilizers, then we are done, so by Fact~\ref{fact.DefPrimImpliesPrim}, $(X,G)$ is primitive. The previous fact now applies. 
\end{proof}

\subsection{Actions of solvable groups}
Theorem~\ref{thm.A} and, more generally, Conjecture~\ref{conj.ProbPGL} are about finding limits on the degree of generic transitivity; the following is in a similar vein but certainly in a  different direction.   

\begin{proposition}[{cf. \cite[Proposition~1]{PoV07}}]\label{prop.TransBoundSolvable}
Let $(X,G)$ be an infinite transitive permutation group of finite Morley rank.
\begin{enumerate}
\item If $G^\circ$ is nilpotent, then $\gtd(X,G) = 1$.
\item If $G^\circ$ is solvable, then $\gtd(X,G) \le 2$.
\end{enumerate} 
\end{proposition}
\begin{proof}
If there is a counterexample to the first point, then we can produce a counterexample $(X,G)$ for which $G$ is connected and the action is definably primitive. Let us elaborate a bit. If  $(X,G)$ is generically $2$-transitive, then $X$ is connected by Fact~\ref{fact.GenericTwoImpliesConnected}, so by Fact~\ref{fact.ConnectedCompGenericTrans}, we may assume that $G$ is connected. Next, replace $X$ by a nontrivial (hence infinite) virtually definably primitive quotient by moving to a  proper definable subgroup of maximal rank that contains a point stabilizer. By Lemma~\ref{lem.QuotientGenericTrans}, moving to a quotient preserves generic $2$-transitivity, and of course, we can ignore the kernel. Now Fact~\ref{fact.primQuotient} applies since generic $2$-transitivity 
implies that point stabilizers are infinite by Remark \ref{generalremarksongenerictransitivity}\eqref{generalremarksongenerictransitivity.ps}, and we obtain our desired counterexample. 
Now fix an $x\in X$. By primitivity, we find that $G = Z(G)\rtimes G_x$ and that the action of $G_x$ on $X-\{x\}$ is equivalent to the action of $G_x$ on $Z(G)$ by conjugation. Certainly, generic $2$-transitivity is impossible. 

For the second point, we consider, as before, a counterexample for which 
we additionally have that $G$ is connected and the action is definably primitive. 
Fix an $x\in X$, and let $Z$ denote the center of the 
Fitting subgroup of $G$. Now we find that $G = Z\rtimes G_x$. 
As $Z$ acts regularly on $X$, $Z$ must be self-centralizing, so $Z$ contains $F(G)$. 
However, this implies, by \cite[I,~Lemma~8.3]{ABC08}, that $G/Z$ is abelian. 
Thus $G_x$ is abelian, so by the first point, the action of $G_x$ on $X-\{x\}$ 
is not generically $2$-transitive. Thus, it cannot be that $(X,G)$ is generically $3$-transitive. 
\end{proof}

\subsection{The $\Sigma$-groups}\label{subsec.Sigma}
The next definition singles out particular subgroups that are essential to the study of generically $n$-transitive actions. 

\begin{definition}
Let $(X,G)$ be a permutation group of finite Morley rank. If $x_1,\ldots,x_n\in X$ are in general position,  define $\Sigma(x_1,\ldots,x_{k-1};x_k,\ldots,x_n)$ to be the subgroup of $G_{\{x_1,\ldots,x_{n}\}}$ that fixes $x_k,\ldots,x_n$ pointwise; when there is no semicolon, $\Sigma(x_1,\ldots,x_n)$ is defined to be $G_{\{x_1,\ldots,x_{n}\}}$. 
\end{definition}

Let us give a couple of examples to keep in mind; the latter is the most relevant for the sequel. 

\begin{example}
Let $K$ be an algebraically closed field. Set $V:=K^3$, and fix an ordered basis $\mathcal{B}:=(e_1,e_2,e_3)$ for $V$.
\begin{enumerate}
\item Consider $\gl_3(K)$ acting naturally on $V$. This action is generically $3$-transitive, and $\Sigma(e_1,e_2,e_3)$ is the group of the permutation matrices (with respect to $\mathcal{B}$). This recovers the Weyl group.
\item Consider $\pgl_3(K)$ acting naturally on $\proj(V)$. This action is generically $4$-transitive with $\Sigma\left(\langle e_1\rangle,\langle e_2\rangle,\langle e_3\rangle;\langle e_1+e_2+e_3\rangle\right)$ equal to (the image in $\pgl_3(K)$ of) the group of permutation matrices. Thus, we again, but in a slightly different way, recover the Weyl group. 
\end{enumerate}
\end{example}

When $(X,G)$ is generically $n$-transitive, $\Sigma(x_1,\ldots,x_n)/G_{x_1,\ldots,x_n}$ is the full symmetric group on $\{x_1,\ldots,x_{n}\}$, and, as in the above examples, if we have a generically \emph{sharply} $n$-transitive action, then $\Sigma(x_1,\ldots,x_n)$ is  equal to $\sym(x_1,\ldots,x_{n})$. This is extremely useful when combined with the following lemma.

\begin{lemma}\label{lem.SigmaGroupLemma}
Let $(X,G)$ be an infinite permutation group of finite Morley rank with $n:=\gtd(X,G)\ge 2$ and  $x_1,\ldots,x_n\in X$ in general position. Then $\Sigma(x_1,\ldots,x_{n-1};x_n)$ acts faithfully on $(G_{x_1,\ldots,x_{n-1}})^\circ$.
\end{lemma}
\begin{proof}
Since $X$ is connected by Fact~\ref{fact.GenericTwoImpliesConnected}, we may apply 
Lemma~\ref{lem.HKlemma} with $H= \Sigma(x_1,\ldots,x_{n-1};x_n)$ and $K=(G_{x_1,\ldots,x_{n-1}})^\circ$.
\end{proof}

Lemma~\ref{lem.SigmaGroupLemma} leads naturally to the consideration of actions of $\sym(n)$, or rather covers of $\sym(n)$, on connected groups of finite Morley rank. Such actions were also considered in \cite{BoCh08}, but the most relevant result for us, namely \cite[Lemma~4.6]{BoCh08}, is not entirely correct as it omits many representations of symmetric groups over the field with $2$-elements, see \cite{DiL08} or the introduction to \cite{DyR79}. However, we will only be concerned with actions of $\sym(n)$ on connected groups of rank at most $2$, so we can quickly prove the little that we need. We will make frequent use of the basic fact that if $\alpha$ is a definable involutory automorphism of a connected group of finite Morley rank $G$ and $C_G(\alpha)$ is finite, then $\alpha$ inverts $G$, see \cite[I,~Lemma~10.3]{ABC08}.

\begin{lemma}\label{lem.symmOnRankOne}
Assume that  $\sym(n)$ acts definably on a connected group $A$ of Morley rank $1$. Then $\alt(n)$ is contained in the kernel. 
\end{lemma}
\begin{proof}
By rank considerations, every transposition of $\sym(n)$ must either fix or invert all of $A$. If some transposition fixes $A$, the kernel contains a transposition, so it is all of $\sym(n)$. Otherwise, every transposition inverts $A$, so $\alt(n)$  fixes $A$.
\end{proof}

\begin{lemma}\label{lem.symmOnRankTwo}
If $\sym(n)$ acts definably and faithfully on a connected group $A$ of Morley rank $2$, then $n \le 3$.
\end{lemma}
\begin{proof}
Towards a contradiction, assume that $n\ge 4$. Choose $\Sigma \le \sym(n)$ with $\Sigma \cong \sym(4)$; note that $\Sigma':=[\Sigma,\Sigma]$ is nonabelian. We work in $A\rtimes \Sigma$.

We first treat the case when $A$ has a definable connected $\Sigma$-invariant subgroup $N$ of rank $1$, i.e. when $A$ is not $\Sigma$-minimal. Now, $A$ is solvable by Fact~\ref{fact.rankTwoGroups}, so we may take $N$ to be normal in $A$. By the previous lemma, $\Sigma'$ centralizes $A/N$ and $N$, so $[[A,\Sigma'],\Sigma'] = 1$. An application of the three subgroup lemma for commutators shows that $[[\Sigma',\Sigma'],A] = 1$, so $[\Sigma',\Sigma']$ acts trivially on $A$. The action is faithful, so $\Sigma'$ is abelian,  a contradiction. 

Now assume that $A$ is $\Sigma$-minimal. Thus, as $A$ is solvable, $A$ is abelian. Let $V < \Sigma$ be the normal subgroup of order $4$, and let $i$, $j$, and $k$ denote the three (commuting) involutions of $V$. Note that $C_A(i)$, $C_A(j)$, and $C_A(k)$ are $\Sigma$-conjugate, so they all have the same rank. If one of these centralizers is finite, then we find that $i$, $j$, and $k$ each invert $A$, which is impossible since  $k = ij$. Since the action is faithful,  $C^\circ_A(i)$, $C^\circ_A(j)$, and $C^\circ_A(k)$ each have rank $1$, and because $A$ is $\Sigma$-minimal, these groups are pairwise distinct. As $A$ has rank $2$, $A = C^\circ_A(i) + C^\circ_A(j)$, and since $k$ normalizes both $C^\circ_A(i)$ and $C^\circ_A(j)$, we find that $k$ inverts $A$. This is a contradiction. 
\end{proof}

\subsection{Actions on sets of rank $1$}
It is hard to overstate the importance of the following fact; it underlies almost every aspect of our proof of Theorem~\ref{thm.A}.  

\begin{fact}[Hrushovski, see \protect{\cite[Theorem~11.98]{BoNe94}}]\label{fact.Hru}
Let $(X,G)$ be a transitive permutation group of finite Morley rank with $X$ connected of rank $1$. Then $\rk G \le 3$, and if $\rk G >1$, there is a definable algebraically closed field $K$ such that either
\begin{enumerate}
\item $(X,G)$ is equivalent to $(K,\aff_1(K))$, or
\item $(X,G)$ is equivalent to $(\proj^1(K),\psl_2(K))$.
\end{enumerate}
\end{fact}

Upon invoking Lemma~\ref{lem.QuotientGenericTrans}, we obtain the following corollary for actions on sets of rank $2$.

\begin{corollary}\label{cor.FourTransVDP}
Let $(X,G)$ be a transitive permutation group of finite Morley rank with $G$ connected and $\rk X = 2$. If $\gtd(X,G)\ge 4$, then the action is virtually definably primitive. 
\end{corollary}
\begin{proof}
Let $(X,G)$ be a transitive permutation group of finite Morley rank with $G$ connected, $\rk X = 2$, and $\gtd(X,G)\ge 4$. 
Assume that the action is not virtually definably primitive. Then by Fact \ref{fact.characterizeDefnPrimi}\eqref{fact.characterizeDefnPrimi.vdp} and the assumption that $\rk X = 2$, 
a point stabilizer $G_x$ is contained in a definable subgroup $H$ of rank $\rk G_x + 1$ and corank $1$ in $G$. 
This implies that there is a (not necessarily faithful) quotient $(\modu{X},G)$ with $\modu{X}$ of rank $1$. As $(X,G)$ is transitive, $(\modu{X},G)$ is also transitive, and $\modu{X}$ has degree $1$. 
Hence, Fact~\ref{fact.Hru} applies. However, Lemma~\ref{lem.QuotientGenericTrans} implies 
that $(\modu{X},G)$ is generically $4$-transitive, but this is impossible. 
\end{proof}

The uninitiated reader may find the proof of Corollary \ref{cor.FourTransVDP} rather heavy. 
It is not. Indeed, unlike ordinary multiple transitivity, generic multiple transitivity does not automatically yield any form of primitivity. Conjecture \ref{conj.ProbPGL}
generalizes the special case illustrated by Corollary \ref{cor.FourTransVDP}
and proposes a setting sufficiently strong to achieve definable primitivity.
We will finish this section with an informative example of a generically $2$-transitive action that is not 
virtually definably primitive. 

\begin{example}
Let $K$ be an algebraically closed field and $G$ the affine group
$K^+\rtimes K^*$. The group $G$ acts on the affine line 
as follows: if $(u,t)\in G$ and $x\in K$, then $(u,t)\cdot\ x = tx+u$. 
Using this action of $G$ on $K$, we now consider the induced coordinatewise action of $G\times G$ on $K\times K$. 
This action is transitive. It is nevertheless not $2$-transitive.
Indeed, the stabilizer of the pair $((x,y),(x',y'))\in K^2\times K^2$, 
is trivial if $x\neq x'$ and $y\neq y'$, while it is isomorphic
to $1\times K^*$ (resp. $K^*\times 1$) if $y=y'$ (resp. $x=x'$). 

On the other hand, the existence of trivial $2$-point stabilizers
shows that the action is generically $2$-transitive. This does not suffice
to conclude that the action is virtually definably primitive. Indeed,
any $1$-point stabilizer in $G\times G$ is conjugate to $K^*\times K^*$,
and this is contained in $G\times K^*$, a proper definable subgroup
of $G\times G$ of rank strictly greater than that of $K^*\times K^*$.

It is worth noting that this example is not an artificial one. 
As the $2$-point stabilizer in the example at the beginning of Section \ref{sec.Aone}
shows, this is in fact a configuration one has to deal with.
\end{example}

\subsection{Moufang sets}\label{subsec.Moufang}
 
We will need a classification result from the theory of Moufang sets of finite Morley rank to put the final touches on the proof of Thoerem~\ref{thm.Aone}. The bare minimum is included below. For more background on Moufang sets, we recommend  \cite{DMeSe09}; specifics on a context of finite Morley rank can be found in \cite{WiJ11}.

\begin{definition}
For a set $X$ with $|X| \ge 3$ and a collection of groups $\{U_x : x\in X\}$ with each $U_x \le \mbox{Sym}(X)$, we say that $(X,\{U_x : x\in X\})$ is a \textdef{Moufang set} if for $G := \langle U_x : x\in X \rangle$ the following conditions hold:
\begin{enumerate}
\item each $U_x$ fixes $x$ and acts regularly on $X\setdiff \{x\}$,
\item $\{U_x : x\in X\}$ is a conjugacy class of subgroups in $G$.
\end{enumerate}
We call $G$ the \textdef{little projective group} of the Moufang set, and each $U_x$ for $x\in X$ is called a \textdef{root group}. The $2$-point stabilizers in $G$ are called the \textdef{Hua subgroups}.
\end{definition}

It is easy to see that $G$ acts $2$-transitively on $X$, and as such, the theory of Moufang sets is about a special class of $2$-transitive permutation groups. This class of $2$-transitive groups includes the sharply $2$-transitive groups, but the motivating example for the subject comes from the natural action of $\psl_2$ on the projective line where the root groups are the unipotent radicals of the Borel subgroups. 

\begin{definition}
We will say that a Moufang set $(X,\{U_x : x\in X\})$ with little projective group $G$ is \textdef{interpretable} in a structure if the root groups, $X$, $G$, and the action of $G$ on $X$ are all interpretable in the structure. Now we define a \textdef{Moufang set of finite Morley rank} to be a Moufang set interpretable in a structure of finite Morley rank. 
\end{definition}

In the end, we will only be interested in Moufang sets of finite Morley rank for which the root groups and the Hua subgroups are both abelian. Using \cite[Proposition 11.61]{BoNe94} when the Hua subgroups are trivial (this is the sharply $2$-transitive case) and using a combination of \cite[Main~Theorem]{SeY09}, \cite[Theorem~6.1]{DMeWe06}, and \cite[Theorem~1.1]{WiJ10} otherwise, one obtains the following classification theorem.

\begin{fact}\label{fact.Moufang}
Let $(X,\{U_x : x\in X\})$ be a Moufang set of finite Morley rank with infinite abelian root groups and abelian Hua subgroups. If $G$ is the little projective group of $\mathbb{M}(U,\tau)$, then there is an algebraically closed field $K$ for which $(X,G)$ is isomorphic to either $(K,\aff_1(K))$ or $(\proj^1(K),\psl_2(K))$.
\end{fact}

\section{Preliminary recognition of $\pgl_3$}\label{sec.TargetRecognition}
The purpose of this section is to provide a general recognition result for $\pgl_3$ that approximates Theorem~\ref{thm.A}. This is where the geometry occurs; our eventual proof of Theorem~\ref{thm.A}, or rather Theorem~\ref{thm.Aone}, will then ``only'' require us to show that the following proposition applies.   

\begin{proposition}\label{prop.Recognition}
Let $(X,G)$ be a generically $4$-transitive permutation group of finite Morley rank with  $\rk X =2$. Further, suppose that
\begin{enumerate}
\item the action is $2$-transitive,
\item the action is not $3$-transitive, and 
\item if $x,y,z\in X$ are in general position, then $\fp((G_{x,y,z})^\circ) = \{x,y,z\}$.
\end{enumerate}
Then $(X,G) \cong (\proj^2(K),\pgl_3(K))$ for some algebraically closed field $K$. 
\end{proposition}

We remark that this proposition should generalize in an obvious way to higher dimensional projective groups, but we have not taken this up here. Our proof builds a projective plane and ultimately relies on the following theorem of Tent and Van Maldeghem, of which we only give a partial account. 

\begin{fact}[{\cite[1.3.~Theorem]{TeVMa02}}]\label{fact.TentVM}
Let $G$ be a group of finite Morley rank acting faithfully on a projective plane $\mathcal{X}$ with strongly minimal point rows and lines pencils. If $G$ acts transitively and definably on the set of ordered ordinary $3$-gons, i.e. ordered triples of noncollinear points, then $(\mathcal{X},G) \cong (\proj^2(K),\pgl_3(K))$ for some algebraically closed field $K$. 
\end{fact}

\subsection{Defining the geometry}
\begin{definition}
Let $(X,G)$ be a generically $3$-transitive permutation group of finite Morley rank. We define a relation $\ell$ on $X^3$ by $\ell(x,y,z)$ if and only if $x,y,z$ are \textbf{not} in general position. We read $\ell(x,y,z)$ as ``$x$, $y$, $z$  are \textdef{collinear}.'' We say  $x$, $y$, $z$, $w$ form a \textdef{$4$-gon} if no three of them are collinear. 
\end{definition}

The orbits of $G$ on $X^3$ are uniformly definable, as are their ranks, so $\ell$ is definable. Note that, by definition, an ordered $3$-gon is the same thing as a triple in $X^3$ that is in general position, so generic $3$-transitivity translates to $G$ acting transitively on ordered $3$-gons. Further, if $(X,G)$ is generically $4$-transitive, then any four points in general position form a $4$-gon, but there is no reason to believe that the points forming a $4$-gon are necessarily in general position. 

\begin{definition}
For $x,y\in X$ in general position, define the \textdef{line through $x$ and $y$} to be $\li{xy}:=\ell(x,y,X)\subset X$. A \textdef{line} will be any subset of $X$ of the form $\li{xy}$ for some $x$ and $y$ in general position, and set of lines is denoted $\lines$.
\end{definition}

Note that  $\li{xy}$ is precisely the union of the nongeneric orbits of $G_{x,y}$.

\subsection{Refining the geometry}
We carry the following setup throughout the current subsection. 

\begin{setup}
Assume that $(X,G)$ is as in Proposition~\ref{prop.Recognition}. The third item from the proposition will be referred to as the \emph{Fixed Point Assumption}. We  call the elements of $X$ points and define the set of lines $\lines$ as above.
\end{setup}

Notice that $2$-transitivity is equivalent to each pair of distinct points being in general position. 

\begin{lemma}
Suppose that $x,y,z\in X$ are three different points. Then $z\in \li{xy}$ if and only if $y\in \li{xz}$ if and only if $x \in \li{yz}$.
\end{lemma}
\begin{proof}
We have that $z\in \li{xy}$ if and only if $x$, $y$, and $z$ are collinear, and as collinearity does not depend on the order of the points, the lemma holds. 
\end{proof}

We now aim to show that two distinct points determine a unique line. The following lemma is the essential step.

\begin{lemma}
Let $x,y\in X$ be distinct. If $z\in \li{xy}$ and $z\neq x$, then $\li{xy} = \li{xz}$.
\end{lemma}
\begin{proof}
Assume $z\in \li{xy}$ with $z\neq x$. If $z=y$, there is nothing to prove, so we also assume that $z \neq y$. First, we consider when there exists an $a\in \li{xz}-\li{xy}$. By $2$-transitivity, we have that $\rk(G_{x,y}/G_{x,y,z,a}) = \rk(G_{x,z}/G_{x,y,z,a})$.  Since $a\in \li{xz}$, we have that $\rk(G_{x,z}/G_{x,z,a}) \le 1$. As  $y\in \li{xz}$ by the previous lemma, we see that $\rk(G_{x,z}/G_{x,z,y}) \le 1$, so $\rk(G_{x,z,a}/G_{x,y,z,a})$ is also at most $1$.  We now see that $\rk(G_{x,y}/G_{x,y,z,a}) \le 2$, but we also have that $\rk(G_{x,y}/G_{x,y,a}) = 2$ since $x,y,a$ are in general position. Thus, $\rk(G_{x,y,a}/G_{x,y,z,a}) =0$, and $z$ is included in a finite orbit of $G_{x,y,a}$ contradicting the Fixed Point Assumption. We conclude that $\li{xz} \subseteq \li{xy}$. By the previous lemma, we may swap the role of $y$ and $z$ in this argument, so we also find that $\li{xy} \subseteq \li{xz}$.
\end{proof}

\begin{proposition}
Let $x,y\in X$ be distinct. Then $\li{xy}$ is the unique line containing $x$ and $y$. Hence, any two lines intersect in at most $1$ point. 
\end{proposition}
\begin{proof}
Clearly $x,y\in \li{xy}.$ Now assume that $x,y \in \li{ab}$ with $a$ and $b$ distinct elements of $X$.  We may assume that $x\neq a$, so by the previous lemma, $\li{ab} = \li{ax}$. Now $y\in \li{ax}$ (since $\li{ax} = \li{ab}$), and $y\neq x$. Thus, $\li{ax} = \li{yx}$. 
\end{proof}

Now, $2$-transitivity ensures that $G$ is transitive on the lines. Notice that a line is definable, so the set of lines may be identified with the definable set $G/G_{\li{}}$ where $G_{\li{}}$ is the setwise stabilizer of some fixed  $\li{}\in \lines$. Also, for a fixed $x\in X$, $G_x$ is transitive on $X-\{x\}$, and $y\sim z$ if and only if $\li{xy}=\li{xz}$ defines an equivalence relation on $X-\{x\}$. Certainly, the line pencil $\lines(x)$ may be identified with the quotient. Using Lemma~\ref{lem.QuotientGenericTrans}, we find that $G_x$ acts transitively and generically $3$-transitively on $\lines(x)$, so in particular, $\lines(x)$ is connected by Fact~\ref{fact.GenericTwoImpliesConnected}. We collect some further details about the geometry.

\begin{lemma}
Let $x,y,z \in X$ be in general position. Then 
\begin{enumerate}
\item the point-row $\points(\li{xy})$, i.e. the set $\li{xy}$, has rank $1$,
\item $\lines(x)$ is strongly minimal, 
\item $G_{x,y,z}$ acts transitively on $\lines(x)-\{\li{xy},\li{xz}\}$.
\end{enumerate}
\end{lemma}
\begin{proof}
If $\rk \points(\li{xy}) \neq 1$, then $\rk\points(\li{xy}) = 0$. In this case, every orbit of $G_{x,y}$ on $\points(\li{xy})$ is finite. By the Fixed Point Assumption, $\li{xy} = \{x,y\}$, and this contradicts our assumption that the action is not $3$-transitive. Thus, $\rk \points(\li{xy}) = 1$. Further, the lines in $\lines(x)$ partition the rank $2$ set $X-\{x\}$, so as each line has rank $1$, we find that $\lines(x)$ also has rank $1$. We already noted that $\lines(x)$ is connected, so it is strongly minimal. 

For the third point, we now have that $G_x$ acts transitively and generically $3$-transitively on the strongly minimal set $\lines(x)$, so $G_x$ induces $\psl_2$ on $\lines(x)$ by Fact~\ref{fact.Hru}. In particular, if $K$ is the kernel of the action of $G_x$ on $\lines(x)$, then $G_x/K$ acts sharply $3$-transitively on $\lines(x)$. Set $T:= G_{x,y,z}$ and $H := G_{\li{xy},\li{xz}}$.  Now, $H$ contains $K$, and by the structure of $\psl_2$, $H/K$ is connected of rank $1$. By the generic $4$-transitivity of $G$ on $X$, we know that  $T$ has a generic orbit on $X-\{x\}$, so as the action of $T$ on $\lines(x)$ can be identified with a quotient of $T$ acting on $X-\{x\}$ (see the discussion preceding the statement of this lemma), this action must be generically transitive as well. Thus, the image of $T$ in $G_x/K$ is infinite, and as $T$ is contained in $H$ with $H/K$ connected of rank $1$, we find that $T$ covers $H$ in the quotient. Since $H$ acts transitively on $\lines(x)-\{\li{xy},\li{xz}\}$, $T$ does as well.
\end{proof}

\begin{proposition}
We have that $(X,\lines)$ is a projective plane.
\end{proposition}
\begin{proof}
Since any $4$ points in general position form a $4$-gon, it only remains to show that every pair of lines intersect. Let $\li{}$ and $\li{}'$ be distinct lines. Fix $x \in \li{}$, and $y,z\in \li{}'$ with $y\neq z$. If $x\in \li{}'$ we are done, so assume that $x\notin \li{}'$. Then $x,y,z$ are in general position, so by the previous lemma, $G_{x,y,z}$ acts transitively on $\lines(x)-\{\li{xy},\li{xz}\}$. Thus, for $w\in \li{}'$ different from $y$ and $z$, we find that $G_{x,y,z}$ moves $\li{xw}$ to $\li{}$ while fixing $\li{}'$ (setwise). Since $\li{xw}$ intersects $\li{}'$, $\li{}$ must intersect $\li{}'$ as well.
\end{proof}

\begin{proof}[Proof of Proposition~\ref{prop.Recognition}]
We now have that $G$ acts on the projective plane $(X,\lines)$, and that the plane has strongly minimal line-pencils. Further, we have already noted that $G$ is transitive on ordered $3$-gons. Thus, in order to apply Fact~\ref{fact.TentVM}, it remains to show that each point-row is strongly minimal.

Fix a line $\li{}$ and a point $x$ not on $\li{}$. Consider the function $\varphi:\points(\li{})\rightarrow \lines(x): y \mapsto \li{xy}$. Since two points determine a unique line, this a definable injective function. Of course, every pair of lines intersect, so  $\varphi$ is surjective. Since  $\lines(x)$ is strongly minimal, $\points(\li{})$ is as well.
\end{proof}

\section{Theorem~\ref{thm.Aone}}\label{sec.Aone}
In this section, we address the case of generic \textdef{sharp$^\circ$} $n$-transitivity; this was defined in Section~\ref{sec.Intro}. 

\begin{MainSubTheorem}\label{thm.Aone}
If $(X,G)$ is a transitive and generically \textdef{sharply$^\circ$} $n$-transitive permutation group of finite Morley rank with $G$ connected and $\rk X = 2$, then
\begin{enumerate}
\item $n$ is at most $4$, and
\item if $n = 4$, then $(X,G)\cong (\proj^2(K),\pgl_{3}(K))$ for some algebraically closed field $K$.
\end{enumerate}
\end{MainSubTheorem}

\begin{setup}
Let $(X,G)$ be a transitive and generically sharply$^\circ$ $n$-transitive permutation group of finite Morley rank with $G$ connected and $\rk X = 2$. Assume that $n \ge 4$, and fix $1,2,3,4,\ldots,n\in X$ in general position. 
\end{setup}

If $\sigma \in \sym(4)$, we write $g\in \sigma$ to mean that $g$ is an element of $G$ acting on $\{1,2,3,4\}$ as $\sigma$; we are thinking of $\sigma$ as an element of $G_{\{1,2,3,4\}}/G_{1,2,3,4}$. If $G_{1,2,3,4}$ is known to be trivial, we write $g=\sigma$ as in this case there is a unique $g$ realizing $\sigma$. For example, by ``let $g \in (12)34$'' we mean ``let $g$ be an element of $G$ swapping $1$ and $2$ and fixing $3$ and $4$.''

To prove Theorem~\ref{thm.Aone}, we aim to apply Proposition~\ref{prop.Recognition}. This does not require that we know $n$ in advance, but in fact, everything will proceed much more smoothly once we know that the action is generically  \emph{sharply} $4$-transitive. Our approach is as follows.

\begin{enumerate}
\item Show that the action is generically sharply $4$-transitive.
\item\label{enum.A1Approach2} Expose the structure of generic $2$ and $3$-point stabilizers. 
\item Show that the action is $2$-transitive. 
\item Prove the Fixed Point Assumption.
\end{enumerate}

In the end, we will also show that the final three points imply that the action is not $3$-transitive.  Item \eqref{enum.A1Approach2}, which occurs as Proposition~\ref{prop.PSStructure}, is an important stepping stone for what follows, so it seems worthwhile to briefly recall the structure of point stabilizers in $\pgl_3$. Let $K$ be an algebraically closed field. Fix an ordered basis $(e_1,e_2,e_3)$ for $K^3$, and set $p_1 = \langle e_1 \rangle$, $p_2 = \langle e_2 \rangle$, $p_3 = \langle e_3 \rangle$, $p_4 = \langle e_1 + e_2 +e_3 \rangle$. If $G =  \pgl_3(K)$, then
\[G_{p_1} = \begin{pmatrix} 1 & 0 & 0\\ * & * & * \\ * & * & * \end{pmatrix}, G_{p_1,p_2} = \begin{pmatrix} 1 & 0 & 0\\ 0 & * & 0 \\ * & * & * \end{pmatrix}, G_{p_1,p_2,p_3} = \begin{pmatrix} 1 & 0 & 0\\ 0 & * & 0 \\ 0 & 0 & * \end{pmatrix}.\]
Thus, $G_{p_1,p_2,p_3}\cong K^\times \times K^\times$ is a maximal torus of $G$, and the $2$-point stabilizer $G_{p_1,p_2}$ is equal to $F(G_{p_1,p_2}) \rtimes G_{p_1,p_2,p_3}$ with  $F(G_{p_1,p_2}) \cong K^+ \oplus K^+$.  Note that $G_{p_1,p_2}$ is connected and solvable, but it is properly contained in the Borel subgroup of lower triangular matrices. 

We begin with an observation which, incidentally, does not require that $G_{1,\ldots,n}$ is finite. It is worth noting that the  proof of the following lemma makes use of \cite[Lemma~5.8]{BoCh08} which in turn uses the classification of the simple groups of finite Morley rank of even and mixed type. It is certainly possible that a more elementary argument would suffice.

\begin{lemma}\label{lem.GenericFourImpliesVDP}
The action is virtually definably primitive, and $G$ has even or odd type. 
\end{lemma}
\begin{proof}
Virtual definable primitivity is given by Corollary~\ref{cor.FourTransVDP}. If $K$ is the kernel of a definably primitive quotient, then as the classes in the quotient are finite, $K^\circ$ centralizes every class. Thus, $K$ is finite, so as $G/K$ must have even or odd type by \cite[Lemma~5.8]{BoCh08}, the same is true of $G$. 
\end{proof}

\subsection{Generic \emph{sharp} $4$-transitivity}\label{subsec.TransBound}
The goal of the present subsection is the following proposition. 

\begin{proposition}\label{prop.transBound}
The action is generically sharply $4$-transitive with $G_{1,2}$ solvable and $G_{1,2,3}$ abelian.
\end{proposition} 

Notice that this proposition also implies that the action is primitive by Corollary~\ref{cor.GenNTransAbelian}. We begin by giving the desired bound on $n$.

\begin{lemma}
We have that $n = 4$. Further, if $(G_{1,2,3})^\circ$ is nonabelian, then it is nonnilpotent in characteristic $3$. 
\end{lemma}
\begin{proof}
Assume  $n \ge 5$. Set $\Sigma:=\Sigma(1,2,3,4;5,\ldots,n)$, $B:= (G_{1,\ldots,n-2})^\circ$, and $H:=(G_{1,\ldots,n-1})^\circ$. Then $\Sigma/G_{1,\ldots,n}\cong \sym(4)$, and  $\Sigma$ is finite since $G_{1,\ldots,n}$ is finite. The finiteness of $G_{1,\ldots,n}$ also implies that $\rk B = 4$  and $\rk H = 2$. 

If $H$ is abelian,  we may use Fact~\ref{fact.GenNTransAbelian} to see that $G_{1,\ldots,n} = 1$. This implies that $\Sigma = \sym(4)$, so Lemma~\ref{lem.symmOnRankTwo} provides a contradiction, via Lemma~\ref{lem.SigmaGroupLemma}. 

Next, if $H$ is nilpotent and nonabelian, then $H$ is $p$-unipotent for some prime $p$ by Fact~\ref{fact.rankTwoGroups}, and we distinguish two cases. If $B$ is nonsolvable, then $H$ embeds into a rank $3$ group of the form $\pssl_2$ by Fact~\ref{fact.genericSharpTwoTransRankTwo}, which is clearly a contradiction. However, if $B$ is solvable, then $H\le F(B)$ (see \cite[I,~Corollary~8.4]{ABC08}), so as the conjugates of $H$ in $B$ generate $B$ by Lemma~\ref{lem.GenericTwoTransPS}, we find that $B$ is nilpotent. Since $B$ acts generically $2$-transitively on $X$, this contradicts Lemma~\ref{prop.TransBoundSolvable}. 

Thus, it remains to treat the case when $H$ is nonnilpotent. We know that $\Sigma$ acts faithfully on $H$, so by Fact~\ref{fact.rankTwoGroups}, $\Sigma$ embeds into $H/Z(H) \cong K^+ \rtimes K^\times$ for some algebraically closed field $K$. Now, $H/Z(H)$ is $2$-step solvable while $\Sigma$ is not, so we conclude that $n = 4$.

For the final point, we now have that $n = 4$, so set $\Sigma:=\Sigma(1,2,3;4)$. If $H=(G_{1,2,3})^\circ$ is nonabelian, we may repeat the above arguments to see that $H$ is nonnilpotent and $\Sigma$ embeds into $H/Z(H)\cong K^+ \rtimes K^\times$. Since $\Sigma$ covers $\sym(3)$, the only way this can happen is if $K$ has characteristic $3$.
\end{proof}

\begin{lemma}
We have that $(G_{1,2})^\circ$ is solvable.
\end{lemma}
\begin{proof}
Assume that  $B:= (G_{1,2})^\circ$ is nonsolvable; set $H:=(G_{1,2,3})^\circ$. By Fact~\ref{fact.genericSharpTwoTransRankTwo}, then there is an algebraically closed field $K$ for which $B =Z(B) \cdot Q$ with  $Z(B)\cong K^\times$ and $Q\cong \pssl_2(K)$. Set $Z:=Z(B)$. As $Z$ is central in $B$,  $Z$ intersects $H$ trivially, and $H$ embeds into $B/Z$. Since $H$ has rank $2$,  this implies that $H$ is nonnilpotent with tori isomorphic to $K^\times$.

Note that $\Sigma(1,2;3,4)/G_{1,2,3,4}$ has order $2$, so by  \cite[I,~Lemma~2.18]{ABC08}, $\Sigma(1,2;3,4)$ contains $2$-elements swapping $1$ and $2$. Let $\alpha$ be a $2$-element of $\Sigma(1,2;3,4) - B$ such that $\alpha^2 \in B$; we do not rule out the possibility that $\alpha\in G_{1,2,3,4}$. Since $H$  is nonnilpotent in  characteristic $3$ and $\alpha$ is a $2$-element acting nontrivially on $H$ (by Lemma~\ref{lem.SigmaGroupLemma}), we may apply Fact~\ref{fact.rankTwoGroups} to see that $\alpha$ must centralize some good torus $T$ in $H$. 

We now claim that $\alpha$ inverts $Z$. If not, then, as $Z$ is connected of rank $1$, $\alpha$ centralizes $TZ$, which is a good torus of rank $2$. Since $\rk X =2$, \cite[Lemma~3.11]{BoCh08} implies that the maximal $2$-torus $S$ of $TZ$ is in fact a maximal $2$-torus of $G$. Further, by Lemma~\ref{lem.GenericFourImpliesVDP}, it must be that $G$ has odd type, so we may apply \cite[Corollary~5.16]{BoCh08} to see that $S$ contains all $2$-elements in $C(S)$. This implies that $\alpha \in S <B$, which contradicts our choice of $\alpha$. We conclude that $\alpha$ inverts $Z$.

Now, as $Z$ covers $B/Q$, we see that $C_B(\alpha) \le Q$. However, $H\cap Q$ has rank $1$ and is normal in $H$, so $T$ is not contained in $Q$. Since $T$ is centralized by $\alpha$,  we have a contradiction.
\end{proof}

The next lemma shows that $G_{1,2,3}$ is abelian-by-finite, so in light of Fact~\ref{fact.GenNTransAbelian}, this will complete the proof of Proposition~\ref{prop.transBound}.

\begin{lemma}\label{lem.ThreePointStabAbelian}
We have that $(G_{1,2,3})^\circ$ is abelian.
\end{lemma}
\begin{proof}
Assume not. Set $H:=(G_{1,2,3})^\circ$, $B:= (G_{1,2})^\circ$, and $P:= (G_{1})^\circ$. Then we know $B$ is solvable and $H$ is nonnilpotent in characteristic $3$. We will use the  unipotent radical of $H$ to build a large unipotent subgroup $U$ of $G$, not contained in any point stabilizer, with the property that $(U\cap P)^\circ \triangleleft P$. Further, $(U\cap P)^\circ$ will be normal in $U$, and $N((U\cap P)^\circ)$ will provide a contradiction to virtual definable primitivity.

Let $U_{1,2,3} := F^\circ(G_{1,2,3})$ be the unipotent radical of $(G_{1,2,3})^\circ$, and similarly define $U_{1,2,4}$, $U_{1,3,4}$, and $U_{2,3,4}$. Set $U:=\langle U_{1,2,3},U_{1,2,4},U_{1,3,4},U_{2,3,4}\rangle$ and $A:= \langle U_{1,2,3},U_{1,2,4}\rangle$. We now work to show that $A$ is abelian and normal in $B$. This will also imply that $U$ is abelian since every pair of subgroups in $\{U_{1,2,3},U_{1,2,4},U_{1,3,4},U_{2,3,4}\}$ generate a subgroup that is conjugate to $A$.

Set $F:=F(B)$, and let $Z$ be the connected center of $F$. Since $B$ is solvable, we have that $A\le F$ by \cite[I,~Lemma~8.36]{ABC08}.  Let $B_3$ be the stabilizer of $3$ in $B$. If $B_3Z = B$, then the orbit of $Z$ on $3$ has rank $2$, and as $Z$ centralizes $U_{1,2,3}$, we find that $U_{1,2,3}=1$ by Lemma~\ref{lem.HKlemma}. Thus $B_3Z$ has rank $3$, so $B_3Z$ determines a rank $1$ quotient of the action of $B$ on its generic orbit (see the discussion preceding Fact~\ref{fact.primQuotient}). As $B$ acts generically $2$-transitively on $X$, Lemma~\ref{lem.QuotientGenericTrans} implies that $B$ acts generically $2$-transitively on the quotient as well. By Fact~\ref{fact.Hru}, the kernel of the quotient must have rank $2$ and, hence, must intersect $H$ in an $H$-normal subgroup of rank $1$. Thus, $U_{1,2,3}$ is in the kernel. Then $U_{1,2,4}$ is as well, and we find that the connected component of the kernel must be $A$. Hence, $A$ is normal in $B$, and since $A$ is a connected rank $2$ nilpotent group with two distinct connected rank $1$ subgroups, we find that $A$ is abelian by Fact~\ref{fact.RankTwoAbelian}.

We now move up to $P$; here we show that $\widetilde{A}:=\langle A, U_{1,3,4}\rangle=F^\circ(P)$. Let $N:=N_{P}(A)$. By our work above, $N\ge\langle B, U_{1,3,4}\rangle$, and since $U_{1,3,4}$ is not in $B$, we find that $N$ has infinite index over $B$. Now, the orbit of $P$ on $2$ is generic, so as $A$ fixes $2$, $A$ is not normal in $P$ by Lemma~\ref{lem.HKlemma}. Thus, $N$ has rank $5$. Let $K$ be the kernel of the action of $P$ on the conjugates of $A$ in $P$; this can be interpreted as a rank $1$ set since $N$ has corank $1$ in $P$. As $P$ acts generically $3$-transitively on its generic orbit, which contains $2$, $3$, and $4$, $P$ realizes every permutation of  $\{2,3,4\}$. Thus, $\widetilde{A}$ normalizes three distinct $P$-conjugates of $A$, so $\widetilde{A} \le K$ by Fact~\ref{fact.Hru}. Since $N$ contains $B$ and the action of $P$ on the cosets of $B$ is generically $3$-transitive by Lemma~\ref{lem.CoverGenericTrans}, we find that the action of $P$ on the conjugates of $A$ is also generically $3$-transitive, so by Fact~\ref{fact.Hru} and rank considerations, $K^\circ = \widetilde{A} = F^\circ(P)$. 

Now we are done. Since, $U_{2,3,4}$ is not contained in $P$, we find that $N(\widetilde{A})$, which contains $G_1$ and $U_{2,3,4}$, has infinite index over $G_1$. As $\widetilde{A}$ fixes $1$, $\widetilde{A}$ is not normal in $G$ by Lemma~\ref{lem.HKlemma}, so as $G$ is connected, we have a contradiction to  virtual definable primitivity.
\end{proof}

\subsection{Structure of point stabilizers}
Next, we work to expose the structure of $G_{1,2}$ and $G_{1,2,3}$; note that Proposition~\ref{prop.transBound} implies that $G_{1,2}$ and $G_{1,2,3}$ are connected groups of rank $4$ and $2$, respectively. We  prove the following.

\begin{proposition}\label{prop.PSStructure}
We have $G_{1,2} =F^\circ(G_{1,2})\rtimes G_{1,2,3}$ with $F^\circ(G_{1,2})$ unipotent and abelian and $G_{1,2,3}$  a self-centralizing (hence maximal) good torus of $G$. 
\end{proposition}

\begin{lemma}
It can not be that $G_{1,2,3}$ has two distinct definably characteristic connected subgroups of rank $1$. 
\end{lemma}
\begin{proof}
Suppose that $A$ and $B$ are distinct definably characteristic connected  rank $1$ subgroups of $G_{1,2,3}$. Let $\Sigma:=\Sigma(1,2,3;4)$. Then $\Sigma=\sym(3)$. By Lemma~\ref{lem.symmOnRankOne}, the commutator subgroup $\Sigma'$ centralizes both $A$ and $B$, so $\Sigma'$ centralizes  $G_{1,2,3}$. However, this contradicts Lemma~\ref{lem.SigmaGroupLemma}.
\end{proof}

\begin{lemma}\label{lem.threePointStab}
We have that $G_{1,2} =F^\circ(G_{1,2})\rtimes G_{1,2,3}$ with $F^\circ(G_{1,2})$ abelian.
\end{lemma}
\begin{proof}
Set $H:=G_{1,2,3}$, $B:= G_{1,2}$, $F:=F^\circ(B)$, and $Z:=Z(F)$. Since $B$ is solvable and nonnilpotent (by Proposition~\ref{prop.TransBoundSolvable}) of rank $4$, we have that $F$ has rank $2$ or $3$ by Fact~\ref{fact.rankFourGroups}. 

We first work to show that $F$ has rank $2$; this will take several steps. Suppose $\rk F = 3$. Since $F$ does not contain $H$ by Lemma~\ref{lem.GenericTwoTransPS}, $B=FH$, and the orbit of $F$ on $3$ is generic. As such, Lemma~\ref{lem.HKlemma} implies that no nontrivial element of $Z$ fixes $3$, so in particular, $(H\cap F)^\circ \cap Z = 1$. We now claim that $(H\cap F)^\circ$ is torsion free. Since divisible torsion is central in any connected nilpotent group of finite Morley rank, $(H\cap F)^\circ$ is not a decent torus. Now, if $(H\cap F)^\circ$ is $p$-unipotent, then as $H$ is abelian and covers the divisible group $B/F$ (see \cite[I,~Lemma~8.3]{ABC08}), we find that $H$ has two distinct definably characteristic connected subgroups of rank $1$. This contradicts the previous lemma, so $(H\cap F)^\circ$ is torsion free. Additionally, since $H\cap F$ is nontrivial and centralized by $Z$, the orbit of $Z$ on $3$ cannot be generic by Lemma~\ref{lem.HKlemma}, so $Z$ has rank $1$. 

Now, $HZ$ determines a rank $1$ quotient of the action of $B$ on its generic orbit. By Lemma~\ref{lem.QuotientGenericTrans} and the solvability of $B$, the quotient must be $2$-transitive, and the kernel $K$ has rank $2$. Let $A_{3}:= (H\cap K)^\circ$ and $A_{4}:=(G_{1,2,4}\cap K)^\circ$. Then both groups have rank $1$, and $K^\circ:=\langle A_3,A_4\rangle$. Since $H$ is abelian and contains the torsion free subgroup $(H\cap F)^\circ$, the previous lemma ensures that $A_3$ is also torsion free. Of course the same is true of $A_4$, so by Fact~\ref{fact.RankTwoAbelian}, we conclude that $K^\circ$ is abelian and torsion free. This also implies that $K^\circ$ is contained in $F$, so $A_3 = (H\cap F)^\circ$.

Next, observe that $B/K^\circ$ is nonnilpotent, so it has a unipotent radical which must be $F/K^\circ$. Now, if $F/K^\circ$ is $p$-unipotent for some prime $p$, then $F$ has a $p$-unipotent subgroup $U$, and we find that $F=K^\circ * U$. This cannot happen since $F$ is nonabelian, so $F/K^\circ$ is torsion free. By Fact~\ref{fact.rankTwoGroups}, the image of $H$ in $B/K^\circ$ contains a $3$-torus, so $H$ contains some $3$-torus $S$. Since $H$ also contains $A_3$, the previous lemma says that the definable hull of $S$ is all of $H$. Also, considering the image of $S$ in $B/K^\circ$, we see that $S$ has Pr\"ufer $3$-rank $1$. Now, $\Sigma:=\Sigma(1,2,3;4)\cong\sym(3)$ has elements of order $3$, and by \cite[I,~Lemma~10.18]{ABC08}, these elements must centralize $S$. But then they centralize $H$, which contradicts Lemma~\ref{lem.SigmaGroupLemma} (and  Fact~\ref{fact.GenNTransAbelian}). Finally, we conclude that $F$ has rank $2$.

We now show that $F$ is abelian; the splitting of $B$ will then follow quickly. If $HF \neq B$, then $HF$ determines a rank $1$ quotient of the action of $B$ on its generic orbit, and the kernel must contain $F$. However, the quotient must be $2$-transitive by Lemma~\ref{lem.QuotientGenericTrans}, and this contradicts the fact that $B/F$ is abelian by \cite[I,~Lemma~8.3]{ABC08}. Thus, $HF=B$, $H\cap F$ is finite, and $F$ has a generic orbit on $3$. Suppose that $F$ is not abelian. Then, by Fact~\ref{fact.rankTwoGroups}, $F$ is $p$-unipotent for some prime $p$, and $Z$ has rank $1$. We return to the quotient determined by $HZ$ with $K$ and $A_3$ defined as before. Note that $K^\circ$ is nonnilpotent since $K^\circ\cap H$ contains $A_3$ while $F\cap H$ is finite. Thus, $Z$ is the unipotent radical of $K^\circ$, and $A_3$ is without $p$-torsion. Further, $[F,A_3] \le (F\cap K)^\circ = Z$, so $A_3$ centralizes $F/Z$. We may use \cite[I,~Proposition~9.9]{ABC08} to lift the centralizer from the quotient and see that $C_F(A_3)$ is  infinite. As the structure of $K^\circ$ shows that $C_F(A_3)$ intersects $Z$ trivially, we find that $F$ contains a nontrivial proper connected subgroup different from $Z$, so $F$ must in fact be abelian by Fact~\ref{fact.RankTwoAbelian}. Further, since $F$ has a generic orbit on $3$, Lemma~\ref{lem.HKlemma} shows that $F\cap H = 1$.
\end{proof}

\begin{proof}[Proof of Proposition~\ref{prop.PSStructure}]
Set $H:=G_{1,2,3}$, $B:= G_{1,2}$, and $F:=F^\circ(B)$.  By Lemma~\ref{lem.threePointStab}, $F$ contains $B'$, so \cite[I,~Lemma~8.3]{ABC08} implies that $H$ is divisible abelian. Also, $F$ has a generic orbit on $3$ since $FH = B$.

We now claim that $H$ is contained in a definable subgroup of $B$ of rank $3$. If not, then the action of $G_{1,2}$ on its generic orbit is virtually definably primitive. In this case, Fact~\ref{fact.VprimRankAbelian} tells us that $B$ acting on its generic orbit is equivalent to $(L,\aff_1(L))$ for some algebraically closed field $L$. Now, $H$ generates a definable field isomorphic to $L$ in $\emorph(F)$, and as $\Sigma(1,2;3,4)$ normalizes both $H$ and $F$, the image of $\Sigma(1,2;3,4)$ in $\emorph(F)$ normalizes this field. Thus, $\Sigma(1,2;3,4)$ acts $L$-linearly on $F$ by \cite[I,~Lemma~4.5]{ABC08}. However, $F$ is $1$-dimensional over $L$, so as $H$ induces all of $L^\times$, there must be a nontrivial element in $H\cdot \Sigma(1,2;3,4)$ centralizing $F$. Since, $F$ has a generic orbit on $3$ and $H\cdot \Sigma(1,2;3,4) < G_3$, this contradicts Lemma~\ref{lem.HKlemma}.

Thus $B$ has a definable rank $3$ subgroup containing $H$, and this determines a quotient of $B$ acting on its generic orbit. Let $K$ be the kernel of the quotient. Since the action of $B$ on its generic orbit is generically $2$-transitive and $B$ is solvable, this quotient must be $2$-transitive, so $HK/K$ is a rank $1$ good torus. Thus, $H/H\cap K$, and hence  $H/(H\cap K)^\circ$, is a rank $1$ good torus. Let $A:=(H\cap K)^\circ$; note that $A$ has rank $1$. Now, $K^\circ$ is certainly not contained in $F$ since $H\cap F = 1$, so $K^\circ$ is nonnilpotent. As $K^\circ$ has rank $2$, we know the structure of $K^\circ$, and it must be that $A$ is a good torus. Now we have that $H$ is a divisible abelian group which is an extension of good tori, so $H$ is a good torus by \cite[I,~Lemma~4.21]{ABC08}. Of course, $H$ is self-centralizing by Fact~\ref{fact.GenNTransAbelian}.

It only remains to show that $F$ is unipotent. If not, then $F$ contains a nontrivial decent torus, and since $F$ is abelian, there is a unique maximal decent torus of $F$. As a decent torus has only finitely many elements of each finite order, the connected group $H$ centralizes the maximal decent torus of $F$, but this contradicts the fact that $H$ is self-centralizing.
\end{proof}

\subsection{$2$-transitivity}
\begin{proposition}\label{prop.Transitivity}
The action of $G$ on $X$ is $2$-transitive.
\end{proposition}
\begin{proof}
Since we already know that this action is definably primitive by Corollary~\ref{cor.GenNTransAbelian}, each $1$-point stabilizer has a unique fixed point, and as the $1$-point stabilizers are also connected, they each have a unique orbit of rank $0$, namely the fixed point. Thus, by generic $2$-transitivity, we need only show that a $1$-point stabilizer has no orbits of rank $1$. Recall that by generic sharp $4$-transitivity, every $1$-point stabilizer has rank $6$.

We first show that \emph{every} $2$-point stabilizer in $G$ is solvable-by-finite, even if the points are not in general position. Assume not. Then there is some $x\in X$ for which the orbit $xG_1$ has rank $1$, so $G_{1,x}$ has rank $5$.  Let $K_1$ be the kernel of the action of $G_1$ on $xG_1$ and $K_x$  the kernel of $G_x$ on $1G_x$. By Fact~\ref{fact.Hru}, $G_1/K_1$ and $G_x/K_x$ both have rank at most $3$, so as every $1$-point stabilizer has rank $6$, $K_1$ and $K_x$ both have rank at least $3$. In particular, $G^\circ_{1,x}/K^\circ_1$ and $G^\circ_{1,x}/K^\circ_x$ both have rank at most $2$. Since connected groups of rank $2$ are solvable, we find that $G^\circ_{1,x}$ is solvable if either $K^\circ_1$ or $K^\circ_x$ is solvable. Thus, we may assume that both $K^\circ_1$ and $K^\circ_x$ are nonsolvable. Next, observe that  $K^\circ_1\neq K^\circ_x$ as otherwise  $N(K^\circ_1)$ would contain both $G_1$ and $G_x$ and force $K^\circ_1$ to be normal in $G$. Thus $K^\circ_1$ and $K^\circ_x$ do not both have rank $5$, so we may assume that $K^\circ_1$ has rank at most $4$. Now, $K^\circ_1$ is nonsolvable of rank $4$, so by the structure of groups of rank at most $4$ given in Section~\ref{sec.Background}, $K^\circ_1$ contains a unique, hence characteristic,  component (subnormal quasisimple subgroup) $Q$, which may of course be equal to $K^\circ_1$. If $K^\circ_1 < K^\circ_x$, then $Q$ would also be a component of $K^\circ_x$ since $K^\circ_x$ normalizes $K^\circ_1$, and we would find that $Q$ is normal in both $G_1$ and $G_x$, hence in $G$. This can not happen, so $A:= (K_1\cap K_x)^\circ < K^\circ_1$. As $K_1,K_x < G_{1,x}$, rank considerations force $A$ to be infinite. Since $A$ is normal in $K^\circ_1$ and $K^\circ_1$ is not solvable, we find that $\rk K^\circ_1 = 4$ and that $A$ has rank $1$ or $3$, again using that connected groups of rank $2$ are solvable. Combining the facts that $\rk K^\circ_1 = 4$, $\rk K^\circ_x \ge 3$, and  $\rk G_{1,x}=5$, we find that $\rk A > 1$. Thus $A$ has rank $3$, and the nonsolvability of $K^\circ_1$ implies that $A$ is also nonsolvable. Hence, $A = Q$, and again this forces $Q$ to be normal in $G$. We conclude that $G^\circ_{1,x}$ is solvable.

Now, towards a contradiction, we assume that the $1$-point stabilizers have rank $1$ orbits. We show that, in this case, the solvable radical of a point stabilizer has rank $3$.  Choose $x\in X$ for which $xG_1$ has rank $1$. With $K_1$ defined as before, $K^\circ_1$ is solvable and normal in $G_1$, so $K^\circ_1 \le R(G_1)$. Further, if  $R(G_1)$ has rank larger than $3$, then $G_1/R(G_1)$ has rank at most $2$, and we find that $G_1$ is solvable. Since $G_1$ is generically $3$-transitive, $G_1$ is not solvable by Proposition~\ref{prop.TransBoundSolvable}, so $R(G_1)$ has rank $3$. Additionally, we can conclude that if $y,z \in X$ are \emph{not} in general position then $G_{y,z}$ contains $R^\circ(G_y)$ and $R^\circ(G_z)$. 

Next, observe that the generic orbits of $G_1$ and $G_x$, with $G_x$ as above, cannot coincide since $G$ acts transitively on $X$ and is generated by $G_1$ and $G_x$ by Lemma~\ref{lem.GenericTwoTransPS}. Thus, there is an element of the generic orbit of $G_1$, which we may take to be $2$, for which $2G_x$ has rank less than $2$. Further, $2\neq x$, so $2G_x$ has rank $1$. Note that $(G_{1,x}\cap G_{2,x})^\circ = G_{1,2}$ by rank considerations. Set $R:=R^\circ(G_x)$. Then, $R < (G_{1,x}\cap G_{2,x})^\circ$ by our previous work. Finally, we throw $G_{3,x}$ into the mix. Observe that $R\cap G_{3,x} \le G_{1,2,3} < G_x$. Now $G_{1,2,3}$ is a self-centralizing good torus, so the $G_x$-conjugates of $G_{1,2,3}$ generate $G_x$, see \cite[IV,~Lemma~1.14]{ABC08}. In particular, $G_{1,2,3}$ is not contained in a proper normal subgroup of $G_x$, so $R\cap G_{3,x} < G_{1,2,3}$. Thus, $R\cap G_{3,x}$ has rank at most $1$, so $R\cdot G^\circ_{3,x}$ has rank at least $3+4-1= 6$. Since $R\cdot G_{3,x} \le G_x$, the only possibility is that $R\cdot G_{3,x}$ has rank $6$, but this implies that $G_x= R\cdot G^\circ_{3,x}$ is solvable. This is our final contradiction. 
\end{proof}

\subsection{Fixed points}
Finally, we prove the Fixed Point Assumption. 
\begin{proposition}\label{prop.FixedPoints}
We have that $\fp(G_{1,2,3}) = \{1,2,3\}$.
\end{proposition}
\begin{proof}
First, assume that we can show $\fp(G_{1,2,3}) = \fp(G_{1,2})\cup\{3\}$. Then this also implies that $\fp(G_{1,2,3}) = \fp(G_{1,3})\cup\{2\}$. Now, if $x\in \fp(G_{1,2,3}) - \{1,2,3\}$, then we find that $G_{1,2}$ and $G_{1,3}$ both fix $x$, and as $G_{1,2}$ and $G_{1,3}$ generate $G_1$ by Lemma~\ref{lem.GenericTwoTransPS}, we conclude that $G_1 = G_x$. However, this is a contradiction since primitivity implies that the $1$-point stabilizers have a unique fixed point.

Thus, it suffices to show that $\fp(G_{1,2,3}) = \fp(G_{1,2})\cup\{3\}$. Towards a contradiction, suppose that there is an $x\in \fp(G_{1,2,3})-\fp(G_{1,2})$ different from $3$. We first show that the orbit $xG_{1,2}$ has rank $1$. Recall that $G_{1,2} = F^\circ(G_{1,2})\rtimes G_{1,2,3}$, so $F^\circ(G_{1,2})$ acts regularly on the generic orbit of $G_{1,2}$ on $X$. Thus, if $x$ is in the generic orbit of $G_{1,2}$, then $G_{1,2,3}$ centralizes some nontrivial $u\in F^\circ(G_{1,2})$. Since $G_{1,2,3}$ is self-centralizing, this is a contradiction, so $\rk xG_{1,2} = 1$ and $\rk G_{1,2,x}=3$.

Since $G_{1,2,x}$ contains the rank $2$ maximal good torus $G_{1,2,3}$, $G_{1,2,x}$ has a \emph{unique} rank $1$ unipotent subgroup, namely $U:=(G_{1,2,x}\cap F(G_{1,2}))^\circ$. By $2$-transitivity, $G_{1,2}\cong G_{1,x}\cong G_{2,x}$, so we also have that $U=(G_{1,2,x}\cap F(G_{1,x}))^\circ=(G_{1,2,x}\cap F(G_{2,x}))^\circ$. Now, $U$ is normal in $G_{1,2,x}$, and as $F(G_{1,2})^\circ$ is abelian, $U$ is also normal in $F(G_{1,2})^\circ$. By rank considerations, we find that $U$ is normal in $G_{1,2}$, and repeating the argument, we conclude that $N(U)$ contains $G_{1,2}$, $G_{1,x}$, and $G_{2,x}$. Let $A:= N^\circ_{G_x}(U)$. Then $A$ contains the nonequal connected rank $4$  subgroups $G_{1,x}$ and $G_{2,x}$, so $A$ has rank at least $5$. Since $U$ fixes points other than $x$, $U$ is not normal in $G_x$, so $A$ has rank equal to $5$.

Now, $A$ contains $G_{1,x}$, so $A$ determines a rank $1$ quotient $\modu{\orbit}$ of $G_x$ acting on $\orbit := X-\{x\}$. Let $K$ be the kernel. Since $G_x$ is generically $3$-transitive on $\orbit$, Lemma~\ref{lem.QuotientGenericTrans} and Fact~\ref{fact.Hru} imply that $G_x$ induces $\psl_2$ on $\modu{\orbit}$. Thus, $A$ fixes, i.e. normalizes, a unique class in $\modu{\orbit}$, namely the class of $1$.  Let $Y_0$ be the class of $1$, and set $Y:=Y_0 \cup \{x\}$. We now aim to show that $A$ acts $3$-transitively on $Y_0$ and that $Y = \fp(U)$. 

By construction of $\modu{\orbit}$ (see the comments preceding Fact~\ref{fact.primQuotient}), $Y_0$ is precisely the orbit of $A$ on $1$. Since the connected group $A$ acts transitively on $Y_0$, $Y$ is connected (of rank $1$). We claim that $2\in Y_0$. Since $K$ is normal in $G_x$ and $G_x$ acts transitively (and faithfully) on $X-\{x\}$, $K$ is not contained in $G_{2,x}$. But $K$ is contained in $A$, so by rank considerations, $A= G_{2,x}K$. Thus, the image of $G_{2,x}$ in the quotient covers $A$, and as the $1$-point stabilizers in $\psl_2$ fix a unique point, $G_{2,x}$ fixes a unique class in $\modu{\orbit}$. Since $G_{2,x} < A$, $G_{2,x}$  fixes $Y_0$, and as $G_{2,x}$ also fixes $\modu{2}$, we conclude that $Y_0 = \modu{2}$. Hence, $2\in Y_0$. We now establish the $3$-transitivity of $A$ on $Y_0$. We know that this action is transitive with $Y_0$ connected of rank $1$, so Fact~\ref{fact.Hru} applies. The stabilizer of $1$ in $A$ is $G_{1,x}$, and since $G_{1,2,x}$ has corank $1$ in $G_{1,x}$, we find that the orbit of $G_{1,x}$ on $2$ is generic in $Y_0$. Thus, the action of  $A$ on $Y_0$ is generically $2$-transitive. Hence by Fact~\ref{fact.Hru}, if the action is not $3$-transitive, then it is sharply $2$-transitive, and $G_{1,2,x}$ coincides with the kernel of this action. This would imply that $G_{1,2,x}$ is normal in $A$, but as $G_{1,2,x}$ contains the self-centralizing good torus $G_{1,2,3}$, this would contradict the fact that $G_{1,2,3}$ is generous in $A$. We conclude that the action of  $A$ on $Y_0$ is $3$-transitive, so $A$ acts on $Y_0$ as $\psl_2$.

Next, we show that $Y = \fp(U)$. Of course $A$ normalizes $U$, and $U$ fixes $1$ and $x$. Thus, since $A$ acts transitively on $Y_0$ with $1\in Y_0$, we have that $Y\subseteq \fp(U)$. Let $N$ be the kernel of the action of $A$ on $Y_0$. Now, $K$ is normal in $A$ of rank $3$, so as $A/N$ is simple (of the form $\psl_2$), it must be that $K\cdot N = A$ or $K^\circ = N^\circ$. However, the latter possibility is impossible since $\rk K = 3$ and $\rk N = 2$. Thus, $N\cap K$ is finite, so as $U \le N$, $U$ is not contained in $K$. Now if $U$ were to fix some $z\notin Y$, then $U$ would fix a class in $\modu{\orbit}$ in addition to $Y_0$. By the structure of $2$-point stabilizers in $\psl_2$, this would imply that $U$ is a good torus modulo $K$. Since $U$ is not contained in $K$ and $U$ is unipotent, this can not happen. Thus $Y = \fp(U)$.

We are now all but done. Consider the action of $N(U)$ on $Y$. By our assumption that $x\notin \fp(G_{1,2})$, $N(U)$ is not contained in $G_x$ since $N(U)$ contains $G_{1,2}$. Since $A< N(U)$ is $3$-transitive on $Y-\{x\}$, we find that the action of $N(U)$ on $Y$ is $4$-transitive. This contradicts Fact~\ref{fact.Hru}.
\end{proof}

\subsection{Assembling the proof}
\begin{proof}[Proof of Thereom~\ref{thm.Aone}]
We now verify that we may apply Proposition~\ref{prop.Recognition}, and for this, it only remains to show that the action is not $3$-transitive. Suppose it is, so every three pairwise distinct points of $X$ are in general position. Consider $P:=G_1$ acting on $Y:=X-\{1\}$. Then $(Y,P)$ is $2$-transitive, and for every pair of distinct $y,z\in Y$, we have that $P_y = F^\circ(P_y) \rtimes P_{y,z}$ with $U_y:=F^\circ(P_y)$ abelian. Thus, it is not hard to see that $\Mouf:=(Y,\{U_y:y\in Y\})$ is a Moufang set of finite Morley rank with little projective group $A:=\langle U_y:y\in Y\rangle \le P$, see Subsection~\ref{subsec.Moufang}. Indeed, the only nontrivial point to check is that each pair $U_y$ and $U_z$ are conjugate, but this follows from the fact that, as $A$ is certainly transitive on $Y$, $A$ can conjugate $P_y$ to $P_z$, hence $F^\circ(P_y)$ to $F^\circ(P_z)$. Since $U_y$ and $P_{y,z}$ are both abelian, $\Mouf$ has abelian root groups as well as abelian Hua subgroups. Thus Fact~\ref{fact.Moufang} applies. If $A$ acts on $Y$ as a $1$-dimensional affine group, then the root groups coincide with the $1$-point stabilizers of this action, which are isomorphic to $L^\times$ for some algebraically closed field $L$. Since each $U_y$ is unipotent by Proposition~\ref{prop.PSStructure}, this is impossible, so $A$ must act $3$-transitively on $Y$. Now we have that $G$ is sharply $4$-transitive, but this is also impossible since a sharply $4$-transitive group must be finite, see \cite{TiJ52} or \cite{HaM54}.
\end{proof}

\section{Theorem~\ref{thm.Atwo}}\label{sec.Atwo}
This section is devoted to the proof of the following theorem. 

\begin{MainSubTheorem}\label{thm.Atwo}
Let $(X,G)$ be a transitive and generically $4$-transitive permutation group of finite Morley rank with $G$ connected and $\rk X = 2$. If a generic $4$-point stabilizer has rank less than $2$, then it has rank $0$.
\end{MainSubTheorem}

We adopt the following setup for the remainder of the section; the goal is to find a contradiction.  

\begin{setup}
Let $(X,G)$ be a transitive and generically $4$-transitive permutation group of finite Morley rank with $G$ connected and $\rk X = 2$. Fix $1,2,3,4\in X$ in general position, and assume that  $S:=(G_{1,2,3,4})^\circ$ has rank $1$. 
\end{setup}

As before, $\Sigma(1,2,3;4)$ acts faithfully on $(G_{1,2,3})^\circ$ by Fact~\ref{lem.SigmaGroupLemma}, and if $\sigma \in \sym(4)$, we write $g\in \sigma$ to mean  $g\in G$ and $g$ realizes $\sigma$. Also note that  $(X,G)$ is virtually definably primitive by Corollary~\ref{cor.FourTransVDP}.

Similar to the proof of Theorem~\ref{thm.Aone}, our analysis begins (and in this case ends) by flushing out a detailed description of $G_{1,2}$. We then use the group $\Sigma(1,2,3,4)$ to drive the resulting configuration to a contradiction.

\subsection{Initial analysis}
\begin{lemma}
We have that $(G_{1,2,3})^\circ$ is solvable. 
\end{lemma}
\begin{proof}
Set $H:=(G_{1,2,3})^\circ$ and $\Sigma:=\Sigma(1,2,3;4)$. As $H$ has rank $3$, we look to Fact~\ref{fact.rankThreeGroups}. First, assume that $H$ is a quasisimple bad group. Since $\Sigma$ has a quotient isomorphic to $\sym(3)$, $\Sigma$ has an involution, say $i$. By \cite[Proposition~13.4]{BoNe94}, simple bad groups do not have involutotory  automorphisms, so $[H,i] \le Z(H)$. Since $H$ is connected, $[H,i]$ is as well, so as $Z(H)$ is finite, we find that $[H,i]=1$. This contradicts the fact that $\Sigma$ acts faithfully on $H$, so $H$ is not a quasisimple bad group.

Thus, if $H$ is not solvable, then $H\cong\pssl_2(K)$ for some algebraically closed field $K$ by Fact~\ref{fact.rankThreeGroups}. In this case, $\Sigma$ acts on $H$ as inner automorphisms (see \cite[II,~Fact~2.25]{ABC08}), so $\Sigma$ embeds into $\psl_2(K)$. Let $\modu{S}$ denote the Zariski closure of $S$ in $\psl_2(K)$. Since $S$ is connected and of rank $1$, $S$ is abelian, so $\modu{S}$ is also abelian. Thus, by the structure of  $\psl_2(K)$, $\modu{S}$ must have Zariski dimension $1$, and since $S$ is connected, $\modu{S}$ is Zariski connected. Thus, $\modu{S}$ lies in a Borel subgroup, which is of the form $K^+\rtimes K^\times$, so as $K^+$ and $K^\times$ are connected in the ambient language, it must be that $S = \modu{S}$. Now, if $S$ is unipotent, then $N_H(S)$ is Borel subgroup. Of course $\Sigma \le N_H(S)$, so $\Sigma/S$, which has a quotient isomorphic to $\sym(3)$, embeds into $K^\times$. This is a contradiction, so $S$ is an algebraic torus. In this case, $|N_H(S)/S|=2$, so again there is no room in $N_H(S)$ for $\Sigma$. We conclude that $H$ is solvable.
\end{proof}

The next lemma is an approximation to the fact, which will be shown later, that $S$ is a decent torus. 

\begin{lemma}\label{lem.RankTwoContainingS}
Every definable rank $2$ subgroup of $(G_{1,2,3})^\circ$ that contains $S$ must contain a nontrivial decent torus.  
\end{lemma}
\begin{proof}
Suppose not. Set $H:=(G_{1,2,3})^\circ$, $B:=(G_{1,2})^\circ$, $P:=(G_{1})^\circ$, and $\Sigma:=\Sigma(1,2,3;4)$. Let $V$ be a connected definable rank $2$ subgroup of $H$ containing $S$ such that $V$ contains no decent torus. Then $V$ is unipotent by Fact~\ref{fact.rankTwoGroups}, and either $V$ is abelian or $S = Z^\circ(V)$. In either case, $V$ centralizes $S$. Since $\Sigma(1,2,3;4)$ acts faithfully on $H$ (by Fact~\ref{lem.SigmaGroupLemma}) and $S < \Sigma(1,2,3;4)$, we find that $V=C^\circ_H(S)$ by rank considerations. Further, by considering the action of $H$ on the rank $1$ coset space $V\backslash H$, the fact that $V$ contains no decent torus implies that $V$ coincides with the connected component of the kernel of this action  by Fact~\ref{fact.Hru}, so $V$ is normal in $H$. If $V$ is nonabelian, then $S$ will be characteristic in $V$ and hence normal in $H$. However, Lemma~\ref{lem.HKlemma} implies that $S$ is \emph{not} normal in $H$, so $V$ is abelian. In summary, $C^\circ_H(S)=V$ is a normal abelian unipotent rank $2$ subgroup of $H$.

We now proceed in a  fashion similar to the proof of Lemma~\ref{lem.ThreePointStabAbelian}. Set $U_{1,2,3}:= C^\circ_{G_{1,2,3}}(S)= C^\circ_H(S)$ with similar definitions for $U_{1,2,4}$, $U_{1,3,4}$, and $U_{2,3,4}$.  Set $U:=\langle U_{1,2,3},U_{1,2,4},U_{1,3,4},U_{2,3,4}\rangle$ and $A:= \langle U_{1,2,3},U_{1,2,4}\rangle$. We first show that $A/S$ is abelian and that $A$ is normal in $B$; this will also imply, using the action of $\Sigma(1,2,3,4)$ on $U$, that $U/S$ is abelian.

Note that $A \le C^\circ_B(S)$ with $\rk A \ge 3$. Since $S$ is not normal in $H$, $H$ is not contained in $A$, so $HA$ has rank strictly larger than the rank of $A$. Thus, 
if $\rk A\ge 4$, then $HA$ is generic in $B$, and $A$ has a generic orbit on $3$. Since $A$ centralizes $S$, this would imply that $S = 1$ by Lemma~\ref{lem.HKlemma}, so $A$ must have rank $3$. Then $A/S$ is a rank $2$ group generated by two distinct connected unipotent subgroups, $A/S$ is abelian by Fact~\ref{fact.RankTwoAbelian}. 

We now show that $A$ is normal in $B$. Consider the action of $B$ on $U_{1,2,3}^B$ (the $B$-conjugates of $U_{1,2,3}$), and let $K$ be the kernel.  By our work above, $N_B(U_{1,2,3})$ contains $H$ and $A$, so $N_B(U_{1,2,3})$ has rank at least $4$. As $U_{1,2,3}$ fixes $3$, Lemma~\ref{lem.HKlemma} ensures that $N_B(U_{1,2,3}) \neq B$. Thus, $N_B(U_{1,2,3})$ has rank $4$, so $U_{1,2,3}^B$ is a connected set of rank $1$. By Fact~\ref{fact.Hru}, $B/K$ has rank $1$, $2$, or $3$. Note that $3$ and $4$ are in the same (generic) orbit of $B$ on $X$, so $U_{1,2,4} \in U_{1,2,3}^B$.  If $B/K$ has rank $1$, then $K$ contains $H$ and $(G_{1,2,4})^\circ$, but this contradicts the fact that $B = \langle H, (G_{1,2,4})^\circ\rangle $, see Lemma~\ref{lem.GenericTwoTransPS}. Thus $B/K$ has rank $2$ or $3$. By the structure of $B/K$ given in Fact~\ref{fact.Hru}, $(N_B(U_{1,2,3})\cap N_B(U_{1,2,4}))^\circ$ either coincides with $K^\circ$ or contains a decent torus. By rank considerations, $A = (N_B(U_{1,2,3})\cap N_B(U_{1,2,4}))^\circ$, so as $A$ is unipotent, we conclude that $A = K^\circ$ and that $A$ is normal in $B$.

Set  $\widetilde{A}:=\langle A, U_{1,3,4}\rangle$. Exactly as in the proof of Lemma~\ref{lem.ThreePointStabAbelian}, we find that $\widetilde{A} = F^\circ(P)$ and conclude that $N(\widetilde{A})$ has infinite index over $G_1$, contradicting virtual definable primitivity.
\end{proof}

\subsection{Structure of point stabilizers}
We begin by studying the generic $3$-point stabilizers.
 
\begin{proposition}\label{prop.StructureThreePointStabilizerA2}
For $H:= (G_{1,2,3})^\circ$ and $\Sigma:=\Sigma(1,2,3;4)$, we have that 
\begin{enumerate}
\item $H=F^\circ(H)\rtimes S$, 
\item $\Sigma$ acts faithfully on $F^\circ(H)$, 
\item $S\cong L^\times$ for some algebraically closed field $L$, and
\item $H\cap G_4 = S$.
\end{enumerate}
\end{proposition}
\begin{proof}
First, suppose that $H$ is nilpotent. In this case, $S$ is not a good torus as otherwise $S$ would be central in $H$, see \cite[Corollary~6.12]{BoNe94}. However, we have already observed that Lemma~\ref{lem.HKlemma} implies that $S$ is not normal in $H$. Thus, $S$ is unipotent. As $S$ is not normal in $H$, the  Normalizer Condition for nilpotent groups of finite Morley rank implies that  $N:=N_H(S)$ has rank $2$; recall that the \textdef{Normalizer Condition} states that every definable subgroup of infinite index in a nilpotent group of finite Morley rank is of infinite index in its normalizer. Since $S < N$, Lemma~\ref{lem.RankTwoContainingS} implies that $N$ contains a nontrivial decent torus. Thus, $S$ must be the only nontrivial unipotent subgroup of $N$, so $S$ is characteristic in $N$. Now, $N$ is normal in $H$ by the Normalizer Condition, so we find that $S$ is normal in $H$. This is a contradiction, so $H$ is nonnilpotent.

Let $F:= F^\circ(H)$. Since $H$ is nonnilpotent, $F$ has rank $2$. If $S<F$, then $S$ is not a good torus as otherwise $S$ would be central in $H$. Thus, as in the previous paragraph, $S$ is characteristic in $F$. This is a contradiction, so $S$ is not in $F$. Thus, $H = F\cdot S$ (we will address $F\cap S$ later). This implies that the orbit of $F$ on $4$ is generic, so $\Sigma$ acts faithfully on $F$ by Lemma~\ref{lem.HKlemma}.

We now claim that $F$ must have some definable $S$-invariant subgroup of rank $1$. If not, we can use \cite[I,~Proposition~4.11]{ABC08} to linearize the action of $S$ on $F$, and find that $\Sigma$ acts (faithfully) on $F$ as a subgroup of the multiplicative group of a field. This is absurd. Thus, $F$  has some definable connected $S$-invariant subgroup $A$ of rank $1$. Further, if $S$ does not centralize $A$, then we can linearize the action of $S$ on $A$ to find that $S\cong L^\times$ for some algebraically closed field $L$. Otherwise, if $S$ does centralize $A$, then $A$ is central in $H$ since $F$ is either abelian or $A=Z^\circ(F)$ by Fact~\ref{fact.RankTwoAbelian}. Now, $H$ is nonnilpotent, so $H/A$ is nonnilpotent. Thus, in this case, the image of $S$ in $H/A$ is $L^\times$ for some algebraically closed field $L$, and this implies that $S\cong L^\times$ (see \cite[Corollary~2.7]{WiJ14a} for example). We conclude that, in any case, $S\cong L^\times$ and that either $A$ or $F/A$ is isomorphic to $L^+$. 

We now show that $F\cap S=1$. First, if $F$ is abelian, then $F\cap S=1$ since $S$ acts faithfully on $F$. Now, if $F$ is nonabelian, then $F$ has exponent $p$ or $p^2$ for some prime $p$ by Fact~\ref{fact.rankTwoGroups}, so the characteristic of $L$ is $p$. Thus, $S$ has no $p$-elements, so we find that $F\cap S=1$.

It only remains to address the final point. Set $H_4:= H\cap G_4$, and note that $S= H^\circ_4$. Since $S$ is not normal in $H$, $C_H(S)$ has rank at most $2$, and consequently,  $C_H(S) \cap F(H)$ has rank at most $1$. As above in the third paragraph, Fact~\ref{fact.RankTwoAbelian} implies that $(C_H(S) \cap F(H))^\circ$ is central in $F^\circ(H)$, hence in $H$. This yields that $C^\circ_H(S) = SZ^\circ(H)$, so $Q:= C^\circ_H(S)$ is abelian. Also, $S$ is a good torus, so $Q$ is of finite index in its normalizer. Since $H$ is solvable, $Q$ is self-normalizing by \cite[Th\'eor\`eme~1.2]{FrO00}, so $H_4\le Q$. Now, $H_4$ acts faithfully on $H$ by Lemma~\ref{lem.HKlemma}, so $H_4\cap Z^\circ(H) = 1$. By rank considerations, $Q = Z^\circ(H) \times H_4$, so $H_4 = S$.
\end{proof}

We obtain the following important corollary.

\begin{corollary}
The point stabilizers $G_1$, $G_{1,2}$, $G_{1,2,3}$, and $G_{1,2,3,4}$ are all connected, and in particular, $\Sigma(1,2,3,4)/S \cong \sym(4)$.
\end{corollary}
\begin{proof}
The first and fourth points from Proposition~\ref{prop.StructureThreePointStabilizerA2} show that the orbit of $F^\circ(G_{1,2,3})$ on $4$ is generic and regular. Thus, everything follows from repeatedly applying Fact~\ref{fact.ConnectedPS}.
\end{proof}

The next lemma and its corollary further clarify the structure of $G_{1,2,3}$.

\begin{lemma}\label{lem.SelfNormalizingA2}
We have that $S$ is self-normalizing in $G_{1,2,3}$. 
\end{lemma}
\begin{proof}
Set $H:= G_{1,2,3}$ and $F:= F^\circ(H)$.  The real work is to show that $S$ is self-centralizing in $H$. Suppose not. Since $C_H(S)$ is connected by \cite[Theorem~1]{AlBu08}, $C_H(S)$ must have rank at least $2$, and as $S$ is not normal in $H$, we find that $C_H(S)$ has rank equal to $2$. By Proposition~\ref{prop.StructureThreePointStabilizerA2}, $Z:=(C_H(S)\cap F)^\circ$ has rank $1$. As in the proof of Proposition~\ref{prop.StructureThreePointStabilizerA2}, we find that $Z = Z^\circ(H)$ and $C_H(S) = SZ$.

Now, $H/Z$ is a nonnilpotent connected group of rank two, and we may apply Fact~\ref{fact.rankTwoGroups}. We find an algebraically closed field $L$ for which $F/Z \cong L^+$ and $S\cong L^\times$. Let $p\ge 0$ be the characteristic of $L$. Note that Fact~\ref{fact.rankTwoGroups} does not preclude the possibility that some nontrivial finite subgroup of $S$ centralizes  $F/Z$.  We deal with this first. If $S$ does not acts faithfully on $F/Z$, then there is an $s\in S$ of prime order $q\neq p$, for which $[F,s] \le Z$. Thus, commutation by $s$ is a homomorphism from $F$ to $Z$. Since, $S$ acts faithfully on $F$, the kernel $C_F(s)$  has rank $1$ and certainly contains $Z$. Also, the image $[F,s]$ must have exponent $q$ since $s$ has order $q$, so we find that $F/C_F(s)$ is an elementary abelian $q$-group. Since $F/C_F(s)$ is a quotient of the connected rank $1$ group $F/Z$, it must be that $F/Z\cong L^+$ is also an elementary abelian $q$-group. As $q\neq p$, we have a contradiction, and we conclude that $S$ acts faithfully on $F/Z$.

Now, let $K$ be the kernel of the action of $\Sigma:=\Sigma(1,2,3;4)$ on $F/Z$. Since $S$ induces all of $L^\times$, we find that $\Sigma = S\times K$, so by the previous corollary, $K\cong \sym(3)$.  Further, $K$ acts on $Z$, so $K'$ centralizes $Z$ by Lemma~\ref{lem.symmOnRankOne}. Let $k\in K$ be of order $3$. Then commutation by $k$ is a homomorphism from $F$ to $Z$, and arguing as in the previous paragraph, we find that $F/Z$ and $Z$ are elementary abelian $3$-groups. Also, note that $Z$ is not centralized by all of $K$ as this would imply that $[[F,K],K] = 1$ and hence (by the three subgroup lemma for commutators) that $K$ is abelian. Thus, some involution $j \in K$ acts nontrivially on $Z$. Then either $j$ inverts $F$ or $F=Z\cdot C_F(k)$. In either case, as $F$ has rank $2$, $F$ is abelian. 

Let $i\in S$ be the unique involution; this exists since we now know the characteristic of $L$ is $3$. Since $F$ is abelian, $[F,i]$ is a connected subgroup of $F$ that is inverted by $i$. It must be that $[F,i]$ has rank $1$ with $F = Z \cdot [F,i]$. Most importantly, $[F,i]$ is $K$-invariant. By Lemma~\ref{lem.symmOnRankOne}, $K'$ centralizes both $Z$ and $[F,i]$, so $K'$ centralizes $F$. This contradicts the fact that $\Sigma$ acts faithfully on $F$ by Lemma~\ref{lem.HKlemma}. Thus, we find that $S$ is self-centralizing in $H$. Further, $S$ is a good torus, so $C_H(S) = S$ is of finite index in its normalizer. Since $H$ is solvable, $S$ is self-normalizing by \cite[Th\'eor\`eme~1.2]{FrO00}.
\end{proof}

\begin{corollary}\label{cor.FittingG123A2}
We have that $F^\circ(G_{1,2,3})$ is abelian and unipotent, and in particular, $F^\circ(G_{1,2,3})$ is self-centralizing in $G$.
\end{corollary}
\begin{proof}
 We will employ many of the same ideas as in the proof of Lemma~\ref{lem.SelfNormalizingA2}. Set $H:= G_{1,2,3}$ and $F:= F^\circ(H)$. First, any decent torus of $F$ would be normal, and hence central, in $H$ by \cite[Theorems~6.8 and 6.9]{BoNe94}. Since $S$ is self-centralizing by Lemma~\ref{lem.SelfNormalizingA2}, $F$ must be unipotent.
 
 Now assume that $F$ is not abelian.  Set  $Z:=Z^\circ(F)$. By Fact~\ref{fact.rankTwoGroups}, $F$ has exponent $p$ or $p^2$ for some prime $p$. Using Lemma~\ref{lem.SelfNormalizingA2}, we may linearize the action of $S$ on $Z$, so there is an algebraically closed field $L$ for which $Z \cong L^+$ and $S\cong L^\times$, though $S$ may have a finite kernel when acting on $Z$.  We now claim that $S$ acts faithfully on $F/Z$. Indeed, assume that some nontrivial $s\in S$ centralizes $F/Z$. By \cite[I,~Proposition~9.9]{ABC08}, $C^\circ_F(s)$ is not contained in $Z$, but this forces $F$ to be abelian by Fact~\ref{fact.RankTwoAbelian}. Thus, $S$ acts faithfully on $F/Z$. Linearizing the action of $S$ on $F/Z$, we find that $\Sigma = S\times K$ with $K$ centralizing $F/Z$ and $K\cong \sym(3)$. Let $k\in K$ be of prime order different from $p$. By \cite[I,~Proposition~9.9]{ABC08}, $C^\circ_F(k)$ is not contained in $Z$, so $F$ is abelian by Fact~\ref{fact.RankTwoAbelian}. Further, $C_H(F)$, which we now know contains $F$, normalizes the generic orbit of $F$ on $X$, so as $F$ acts regularly on this orbit, $F = C(F)$.  
\end{proof}

We now address $G_{1,2}$. The following proposition will be the final ingredient needed to prove Theorem~\ref{thm.Atwo}. 

\begin{proposition}\label{prop.BmodFB}
We have that $G_{1,2}/F^\circ(G_{1,2})$ acts faithfully on $F^\circ(G_{1,2})$ with the action equivalent to the natural action of $\ssl_2(L)$ on $L^2$ for some algebraically closed field $L$. Further, if $\beta \in 12(34)$, then $\beta \notin F^\circ(G_{1,2})$
\end{proposition} 
\begin{proof}
Set $B:= G_{1,2}$, and let $Y$ be the generic orbit of $B$ on $X$. We want understand the action of $B$ on $F^\circ(B)$, and we begin by studying the action of $B$ on $Y$. We will use $B_3$ for $G_{1,2,3}$. Note that $\rk B = 5$ and $\rk B_3 =3$. 

We first show that  $(Y,B)$ has a rank $1$ quotient, i.e. that $(Y,B)$ is not virtually definably primitive. This will take some work. Towards a contradiction, assume that $(Y,B)$ is virtually definably primitive. We aim to push this configuration into the realm of Moufang sets (see Subsection~\ref{subsec.Moufang}), so in particular, we will show that $(Y,B)$ is $2$-transitive. We first claim that $(Y,B)$ is definably primitive. Let $(\modu{Y},B)$ be a definably primitive quotient of $(Y,B)$ with kernel $K$. Since the classes in the quotient are finite, $K^\circ$ is trivial, so $K$ is finite and central in $B$. Further, $F^\circ(B_3)$ is self-centralizing in $G$ by the previous corollary, so $K$ is contained in $B_3$. As $(Y,B)$ is transitive and faithful by Fact~\ref{fact.FixGenericSet}, it must be that $K$ is trivial. Since the classes in $\modu{Y}$ are finite and $F^\circ(B_3)$ has a regular and generic orbit on $Y$, we find that $F^\circ(B_3)$ has a regular and generic orbit on $\modu{Y}$. By Fact~\ref{fact.ConnectedPS}, the point stabilizers in $(\modu{Y},B)$ are connected, so as there is no kernel, the point stabilizers from $(\modu{Y},B)$ and $(Y,B)$ coincide. Thus, $(Y,B)$ is definably primitive. 

To show that $(Y,B)$ is $2$-transitive, it remains to show that $B_3$ has no rank $1$ orbits on $Y$. Towards a contradiction, suppose that the orbit of $B_3$ on $y \in Y$ has rank $1$. Set $A:= (B_{3,y})^\circ$. Then $A$ has rank $2$. If $A$ is unipotent, then the structure of $B_3 = G_{1,2,3}$ implies that $A = F^\circ(B_3)$, and as $B_y$ is $B$-conjugate to $B_3$, this also implies that $A = F^\circ(B_y)$. Thus, if $A$ is unipotent, $N_B(A)$ contains $\langle B_3,B_y\rangle$ and thus coincides with $B$, contradicting the faithfulness of $(Y,B)$. We conclude that $A$ contains a good torus, which we know to be self-centralizing. Hence,  $A$ is nonnilpotent, and  $A$ has a \emph{unique} normal definable connected subgroup of rank $1$. Now let $K_3$ be the kernel of the action of $B_3$ on its obit containing $y$, and let $K_y$ be the kernel of the action of $B_y$ on its orbit containing $3$. As $B_3$ is solvable, we may apply Fact~\ref{fact.Hru} to see that either $A = K^\circ_3$ or $K^\circ_3$ is a normal rank $1$ subgroup of $A$; a similar statement holds for $K_y$. Our previous observations now imply that either $K^\circ_3 = K^\circ_y$ or one is characteristic in the other. In either case, we find a nontrivial definable subgroup that fixes $3$ and whose normalizer contains $\langle B_3,B_y\rangle$. As before, this contradicts the faithfulness of $(Y,B)$. We conclude that $B_3$ has no rank $1$ orbits, so $(Y,B)$ is $2$-transitive.

For each $y\in Y$, we now have that $F^\circ(B_y)$ fixes $y$ and acts regularly on $Y-\{y\}$, so the collection $\{F^\circ(B_y) : y\in Y\}$ generates a group $M \le B$ that is associated to a Moufang set with abelian root groups. Further, since $B_{3,4} = S$ is abelian, the Moufang set has abelian Hua subgroups, so we may apply Fact~\ref{fact.Moufang}. If $(Y,M) \cong (L, \agl_1(L))$ for some algebraically closed field $L$, then $F^\circ(B_3)$ is isomorphic to a point stabilizer in $(L, \agl_1(L))$. This implies that  $F^\circ(B_3) \cong L^\times$, but this cannot happen since $F^\circ(B_3)$ is unipotent. Thus, $(Y,M) \cong (L, \psl_2(L))$ for some algebraically closed field $L$. Under this isomorphism, $F^\circ(B_3)$ goes to the unipotent radical of a Borel subgroup of $\psl_2(L)$, so $\rk M = 3\rk F^\circ(B_3)$. Since $F^\circ(B_3)$ has rank $2$ and $B$ has rank $5$, we have a contradiction. Finally, we   conclude that $(Y,B)$ is \emph{not} virtually definably primitive.

Now let $(\modu{Y},B)$ be any rank $1$ quotient of $(Y,B)$, let $K$ be the kernel, and let $\modu{3}$ be the class containing $3$ in $\modu{Y}$. As $(Y,B)$ is generically $2$-transitive, Lemma~\ref{lem.QuotientGenericTrans} implies that the same is true of $(\modu{Y},B)$, so by Fact~\ref{fact.Hru}, $B_{\modu{3}}/K$ must contain a good torus. Now, by the faithfulness of $(Y,B)$, $K^\circ$ is not contained in $B_3$, so by rank considerations, $B_3$ covers $B_{\modu{3}}/K$. Thus, the structure of $B_3 = G_{1,2,3}$ implies that $S$ is not in $K$, so as  $S \le B_{\modu{3},\modu{4}}$, we find that $(\modu{Y},B)$ must be $3$-transitive, again by Fact~\ref{fact.Hru}. Further, we see that $(K\cap B_3)^\circ$ is unipotent. Thus, $K^\circ$ has rank $2$ and is generated by distinct rank $1$ unipotent subgroups, namely  $(K\cap B_3)^\circ$ and $(K\cap B_4)^\circ$, so $K^\circ$ is abelian and unipotent by Fact~\ref{fact.RankTwoAbelian}. Consequently, as $B/K$ is of the form $\psl_2$, it must be that $K^\circ = F^\circ(B)$. Further, if $\beta \in 12(34)$, then $\beta$ swaps $\modu{3}$ and $\modu{4}$ in $\modu{Y}$, so $\beta \notin K$. 

Since $B/K$ is of the form $\psl_2$ with $K/F^\circ(B)$ finite, $B/F^\circ(B)$ is a perfect central extension of $\psl_2$, so it must be that $B/F^\circ(B)$ is of the form $\pssl_2$. Certainly $B/F^\circ(B)$ acts nontrivially on $F^\circ(B)$, so \cite[Fact~2.5]{DeA09} ensures that, regardless of the characteristic, $B/F^\circ(B)$ is of the form $\ssl_2$. Finally, \cite[Fact~2.7]{DeA09} identifies the action of $B/F^\circ(B)$ on $F^\circ(B)$ as the natural one, which of course also implies that the action is faithful.
\end{proof}

\subsection{Assembling the proof}
\begin{proof}[Proof of Theorem~\ref{thm.Atwo}]
Set $H:= G_{1,2,3}$, $B:= G_{1,2}$, and $U:=F^\circ(B)$. Invoking Proposition~\ref{prop.BmodFB}, $B/U$ acts faithfully on $U$ as $\ssl_2(L)$ acts on $L^2$ for some algebraically closed field $L$. 

We first treat the characteristic not $2$ case; assume that $S$ has an involution $i$. Then the image of $i$  in the quotient $B/U$ is nontrivial and central. Also, since the characteristic is not $2$ and $F^\circ(H)$ is unipotent, $F^\circ(H)$ has no involutions. Thus, $F^\circ(H) = V^+ \oplus V^-$ where $V^+:=C_{F^\circ(H)}(i)$ and $V^- := [F^\circ(H),i]$, see \cite[I~Lemma~10.4]{ABC08}. Since $i$ is central in $B/U$, $V^-$ is not all of $F^\circ(H)$, so it must be that $V^+$ and $V^-$ are both connected of rank $1$. Now, $i$ is the unique involution of $S$, so $\Sigma:=\Sigma(1,2,3;4)$ must centralize $i$. Thus, $\Sigma$ acts on both $V^+$ and $V^-$. We can use $S$ to linearize the action of $\Sigma$ on $V^-$; let $K$ be the kernel. As $S$ covers $\Sigma/K$, $\sym(3) \cong SK/S \cong K/K\cap S$, and $K\cap S$ has no involutions. Further, $K$ is finite, so $S$ centralizes $K$. Thus, $K$ is a central extension of $\sym(3)$ by $K\cap S$. Since $K\cap S$ has no involutions, it is not hard to build a subgroup $K_0 \le K$ such that $K = K_0 \times K\cap S$, which implies that $K_0 \cong \sym(3)$. Of course, $K_0$ also acts on $V^+$, so by Lemma~\ref{lem.symmOnRankOne}, there are nontrivial elements of $K_0$ fixing $V^+$. However, such elements centralize all of $F$, but as $F$ has a generic orbit on $4$, this contradicts Lemma~\ref{lem.HKlemma}.

Thus, we now consider when $S\cong L^\times$ for some field $L$ of characteristic $2$. Since $S$ has rank $1$, the centralizer of $S$ in $\Sigma(1,2,3,4)$ has index at most $2$ by Lemma~\ref{lem.symmOnRankOne}. We claim that the index is $2$. Indeed, choose $\beta \in 12(34)$. Note that as $S$ has no involutions, we may take $\beta$ to be an involution by \cite[I,~Lemma~2.18]{ABC08}. By Proposition~\ref{prop.BmodFB}, the image of $\beta$ in $B/U$ is nontrivial, and it certainly normalizes the image of $S$, the latter being a maximal torus of $B/U$. Since $B/U$ is of the form $\psl_2$ in characteristic $2$, we find that $\beta$ is not in $S$, so $\beta$ does not centralize $S$. We conclude that the transpositions in $\Sigma(1,2,3,4)/S$ invert $S$. In particular, any  $\alpha\in (12)34$ inverts $S$.  Again, we may take $\alpha$ to be an involution. Now, $B/U$ is of the form $\psl_2$, so the action of $\alpha$ on $B/U$ is inner. Since $\alpha$ fixes $3$ and $4$, $\alpha$ normalizes two distinct Borel subgroups of $B/U$, namely the images of $G_{1,2,3}$ and $G_{1,2,4}$. Thus, $\alpha$ acts on $B/U$ as an element of a torus. As $\alpha$ is an involution and the characteristic is $2$, $\alpha$ must act trivially on $B/U$. We conclude that $[B,\alpha]\le U$, so $S = [S,\alpha] \le U$. This is a contradiction. 
\end{proof}

\section{Theorem~\ref{thm.A}}\label{sec.A}
We now put the pieces together. We begin with a reduction to the connected case. By the \emph{connected version of Theorem~\ref{thm.A}}, we mean Theorem~\ref{thm.A} together with the additional hypothesis that $G$ is connected. 

\begin{lemma}\label{lem.ReduceToConnected}
Theorem~\ref{thm.A} follows from the connected version of Theorem~\ref{thm.A}.
\end{lemma}
\begin{proof}
Suppose that the connected version of Theorem~\ref{thm.A} has been proven. Let $(X,G)$ be a transitive and generically $4$-transitive permutation group of finite Morley rank with $\rk X = 2$. The punchline is that the connected version of Theorem~\ref{thm.A} applies to $(X,G^\circ)$, and from this, we can easily show that Proposition~\ref{prop.Recognition} applies to $(X,G)$. 

First, by Fact~\ref{fact.GenericTwoImpliesConnected}, $X$ is connected. Then, Lemma~\ref{fact.ConnectedCompNTrans} and Fact~\ref{fact.ConnectedCompGenericTrans} imply that $(X,G^\circ)$ is transitive and generically $4$-transitive.  Thus, we may apply the connected version of Theorem~\ref{thm.A} to see that $(X,G^\circ)\cong (\proj^2(K),\pgl_{3}(K))$ for some algebraically closed field $K$. We now check that Proposition~\ref{prop.Recognition} applies to $(X,G)$. Since $(X,G^\circ)\cong (\proj^2(K),\pgl_{3}(K))$, $(X,G^\circ)$ is $2$ but not $3$-transitive, so by Lemma~\ref{fact.ConnectedCompNTrans}, the same is true of $(X,G)$. It remains to verify the Fixed Point Assumption. Let $x,y,z\in X$ be in general position. Using that $(X,G^\circ)\cong (\proj^2(K),\pgl_{3}(K))$, we find that $(G^\circ)_{x,y,z}$ is connected, so $(G^\circ)_{x,y,z} = (G_{x,y,z})^\circ$. Further, $\fp((G^\circ)_{x,y,z})) = \{x,y,z\}$, so we may indeed apply  Proposition~\ref{prop.Recognition}  to $(X,G)$.
\end{proof}

We now make precise the link between the connected version of Theorem~\ref{thm.A} and Theorems~\ref{thm.Aone} and \ref{thm.Atwo}. Of course, Theorem~\ref{thm.Atwo}, says that the second case in the following lemma is impossible.

\begin{lemma}
Let $(X,G)$ be a transitive permutation group of finite Morley rank with $G$ connected and $\rk X = 2$. If $n:=\gtd(X,G)\ge 4$, then either 
\begin{enumerate}
\item $(X,G)$ is generically sharply$^\circ$ $n$-transitive, or
\item there exists a transitive permutation group of finite Morley rank $(Y,H)$ with $H$ connected and $\rk Y = 2$ for which $\gtd(Y,H)= 4$ and the generic $4$-point stabilizers have rank $1$.
\end{enumerate}
\end{lemma}
\begin{proof}
Since $n\ge 4$, $(X,G)$ is virtually definably primitive by Corollary~\ref{cor.FourTransVDP}. By Lemma~\ref{lem.PrimitiveBoundG}, a generic $n$-point stabilizer has rank at most $1$. If such a stabilizer has rank $0$, $(X,G)$ is generically sharply$^\circ$ $n$-transitive by definition. Otherwise, let $H$ be the connected component of a generic $(n-4)$-point stabilizer, and consider the action of $H$ on its generic orbit in $X$, using Fact~\ref{fact.ConnectedCompGenericTrans}.  
\end{proof}

\begin{proof}[Proof of Theorem~\ref{thm.A}]
By Lemma~\ref{lem.ReduceToConnected}, we may assume that $G$ is connected. Now we apply the previous lemma in conjunction with Theorem~\ref{thm.Atwo} to see that $(X,G)$ is generically sharply$^\circ$ $n$-transitive with $n\ge 4$. Theorem~\ref{thm.Aone} finishes things off.
\end{proof}

\section*{Acknowledgments}
The second author is indebted to Katrin Tent for several helpful discussions about the geometric aspects of the proof; they made everything else possible. He would also like to acknowledge the warm hospitality of Universit\"at M\"unster as well as Universit\'e Claude Bernard Lyon-1 where the results of the article where obtained.

\bibliographystyle{alpha}
\bibliography{WisconsBib}
\end{document}